\documentclass[sn-mathphys,Numbered]{sn-jnl}

\usepackage{graphicx}%
\usepackage{multirow}%
\usepackage{amsmath,amssymb,amsfonts}%
\usepackage{amsthm}%
\usepackage{mathrsfs}%
\usepackage[title]{appendix}%
\usepackage{xcolor}%
\usepackage{textcomp}%
\usepackage{manyfoot}%
\usepackage{booktabs}%
\usepackage{algorithm}%
\usepackage{algorithmicx}%
\usepackage{algpseudocode}%
\usepackage{listings}%
\usepackage{slashed}

\usepackage{lipsum}
\usepackage{color}
\usepackage{epstopdf}

\usepackage{subcaption,caption}
\usepackage{float} 
\usepackage{footmisc}

\graphicspath{{./figures/} }

\ifpdf
  \DeclareGraphicsExtensions{-eps-converted-to.pdf,.pdf,.png,.jpg}
\else
  \DeclareGraphicsExtensions{-eps-converted-to.pdf}
\fi

\theoremstyle{thmstyleone}%
\newtheorem{theorem}{Theorem}[section]
\newtheorem{lemma}{Lemma}[section]
\newtheorem{corollary}{Corollary}[section]
\newtheorem{proposition}{Proposition}[section]%
\newtheorem{assumption}{Assumption}[section]

\theoremstyle{thmstyletwo}%
\newtheorem{example}{Example}[section]%
\newtheorem{remark}{Remark}[section]%

\theoremstyle{thmstylethree}%
\newtheorem{definition}{Definition}[section]%

\numberwithin{equation}{section}

\def\R{\mathbb{R}}
\def\H{\mathbf{H}}
\def\M{\mathbf{M}}
\def\A{\mathbf{A}}

\def\CS{\mathcal{S}}
\def\S{\mathbb{X}^*}
\def\N{\mathbb{N}^*}

\def\bx{{\bf x}}

\def\bz{{\bf w}}
\def\bu{{\bf u}}
\def\bd{{\bf d}}
\def\bv{{\bf v}}
\def\bw{{\bf w}}
\def\bb{{\bf b }}
\def\ba{{\bf a }}

\def\bfz{{\bf z }}

\def\co{{\rm co }}

\def\bg{{\boldsymbol g }}

\def\supp{{\rm supp }}
\def\sign{{\rm sgn}}

\def\prox{{\rm Prox}_{\alpha\lambda\|\cdot\|^q_q}}

\def\eqspace{\arraycolsep=1.5pt\def\arraystretch}

\begin{document}
\flushbottom
\title[Revisiting Lq Norm Regularized Optimization]{\vspace{-1.75cm}
Revisiting $L_{q}(0\leq {q}<1)$ Norm 
Regularized Optimization
}


\author*[1]{\fnm{Shenglong} \sur{Zhou}}\email{szhou2021@163.com}

\author[2]{\fnm{Xianchao} \sur{Xiu}}\email{xcxiu@shu.edu.cn}

\author[3]{\fnm{Yingnan} \sur{Wang}}\email{wyn1982@hotmail.com}

\author[4]{\fnm{Dingtao} \sur{Peng}}\email{dingtaopeng@126.com}

\affil*[1]{
\orgdiv{School of Mathematics and Statistics}, \orgname{Beijing Jiaotong University}, \orgaddress{\city{Beijing}, \country{China}}}

\affil[2]{
\orgdiv{School of Mechatronic Engineering and Automation}, \orgname{Shanghai University}, 
 \orgaddress{\city{Shanghai}, \country{China}}}

\affil[3]{
 \orgname{Southwest Forestry University}, \orgaddress{\city{Yunnan}, \country{China}}}

\affil[4]{\orgdiv{School of Mathematics and Statistics}, \orgname{Guizhou University}, \orgaddress{\city{Guizhou}, \country{China}}}


\abstract{Sparse optimization has seen its advances in recent decades. For scenarios where the true sparsity is unknown, regularization turns out to be a promising solution. Two popular non-convex regularizations are the so-called $L_0$ norm and $L_q$ norm with $q\in(0,1)$, giving rise to extensive research on their induced optimization. However, the majority of these work centered around the main function that is twice continuously differentiable, and the best convergence rate for an algorithm solving the optimization with $q\in(0,1)$ is superlinear. 
This paper explores the $L_q$ norm regularized optimization in a unified way for any $q\in[0,1)$, where the main function has a semismooth gradient. In particular,  we establish the first-order and the second-order optimality conditions under mild assumptions and then integrate the proximal operator and semismooth Newton method to develop two proximal semismooth Newton pursuit algorithms.  Under the second sufficient condition, the whole sequence generated by two algorithms converges to a unique local minimizer. Moreover, the convergence is superlinear and quadratic if the gradient of the main function is semismooth and strongly semismooth at the local minimizer, respectively. Hence,  this paper accomplishes the quadratic rate for an algorithm designed to solve the $L_q$ norm regularization problem for any $q\in(0,1)$.  Finally, some numerical experiments have showcased their nice performance when compared with several existing solvers. }




\maketitle
\section{Introduction}
We study the following sparse optimization with the $L_q$ norm regularization  (SOLq):
\begin{equation}\label{rso} 
\min_{\bx\in\R^n}~F(\bx):= f(\bx) +   \lambda \|\bx\|^q_q, 
\tag{SOLq}
\end{equation} 
where main objective function $f:\R^n\mapsto\R$ is continuously differentiable and bounded from below,    $\lambda>0$ is a penalty parameter, and $\|\bx\|^q_q:=\sum_{i}|x_i|^q$ is the so-called $L_q$ norm with $q\in[0,1)$. Particularly, it is $L_0$ norm when $q=0$ , namely, $\|\bx\|^q_q=\|\bx\|_0$ that counts the number of non-zero elements of $\bx$. We point out that the $L_q$ norm with $q\in[0,1)$ does not adhere to the triangle inequality property, and thus it does not fully conform to the definition of a norm. Typical applications of the above optimization include compressive sensing (CS) \cite{candes2005decoding, candes2006robust, donoho2006compressed,foucart2011hard},    sparse logistic regression \cite{bahmani2013greedy,wangaextended}, quadratic CS or phase retrieval \cite{shechtman2011sparsity,Beck13,shechtman2014gespar}, among others. 

\subsection{Prior arts}
There are fruitful theoretical and numerical results in the realm of sparse optimization, therefore we may only concentrate on the algorithms specifically tailored for addressing problem \eqref{rso} in order to provide motivation for our work.

{\bf a) Case $q=0$.} In this case,  \eqref{rso} is usually phrased as zero norm or cardinality regularized optimization. We review the relevant work from two perspectives.

 For the CS problems,  the early work can be traced back to  \cite{blumensath2008iterative} in which an iterative hard-thresholding algorithm has been proposed.  In  \cite{ito2013variational}, the authors designed a primal-dual active set algorithm based on the augmented Lagrange algorithm. Similar ideas were further extended by  \cite{jiao2015primal, huang2018constructive}, leading to the development of the primal-dual active set continuation ({PDASC}, \cite{jiao2015primal}) algorithm and support detection and root finding ({SDAR}, \cite{huang2018constructive})  algorithm. They leveraged the proximal operator of the $L_0$ norm to determine an active set, on which  Newton steps were executed.

 For the more general \eqref{rso} with $q=0$,  the work in \cite{zhou2021newton} has also combined the proximal operator of the $L_0$ norm with Newton steps to develop a Newton algorithm for the $L_0$ regularized optimization. However, the dominance of Newton steps and the careful selection of several parameters were necessary to ensure convergence. Additionally, in \cite{li2022difference}, the $L_0$ norm regularization was applied to group sparse optimization which was addressed by the difference-of-convex algorithms.

{\bf b) Case $q\in(0,1)$.} Despite considerable interest devoted to investigating \eqref{rso} with $q\in(0,1)$, the majority of work has been carried out for the CS problems. 

 For the CS problems,   it is demonstrated in \cite{krishnan2009fast} that the $L_{1/2}$ and
$L_{2/3}$ regularizations were relatively efficient when applied to image deconvolution. In \cite{chen2010lower}, lower bounds for the magnitudes of non-zero entries in every local optimal solution of model \eqref{rso} were achieved, enabling zero entries of numerical solutions to be identified precisely. Based on the lower bound theory, the authors proposed a hybrid orthogonal matching pursuit-smoothing gradient algorithm. Furthermore, the authors in \cite{Chen2011ComplexityOU} offered an in-depth study of the model's complexity, including its strong NP-hardness and sparsity level of solutions achieved through selecting parameters $\lambda$ and $q$ carefully. 
In \cite{lai2011unconstrained},  a smoothing function was created to approximate the $L_q$ norm, allowing for the exploitation of an iterative reweighted algorithm. A similar idea was then employed into \cite{lai2013improved}. Additionally, special attention was given to the case of $q=1/2$ in \cite{xu2012}, resulting in a half-thresholding algorithm. A comprehensive convergence analysis for this algorithm was delivered in \cite{zeng2014convergence}. Fairly recently, the authors in \cite{hu2021linear} established equivalences among a local minimum, second-order optimality condition, and second-order growth property. Apart from the optimality analysis, they also discussed the linear convergence rates of several inexact gradient algorithms.

  For the general \eqref{rso} with $q\in(0,1)$, the author in \cite{lu2014iterative} exploited a locally Lipschitz function to approximate the $L_q$ norm so as to develop two iterative reweighted minimization methods. Recently,  a hybrid approach combining the proximal gradient and subspace regularized Newton method was cast in \cite{wu2022globally}, where the Kurdyka-$\L{}$ojasiewicz (KL) property was used to establish some preferable convergence properties, such as the sequence convergence in a superlinear rate, see Table \ref{tab:com-algs}.

 \begin{table}[!th]
\renewcommand{\arraystretch}{1}\addtolength{\tabcolsep}{1.5pt}
	\caption{Comparisons of the convergence results for different algorithms. }
	\label{tab:com-algs}
\begin{tabular}{lll rrr|rr } \toprule
&Probs. &Algs.  &Ref.& Global &	  Sequence~~   & Rates of  &	Assumps.	  \\   
& & & & conv. &	  conv.  & conv. &	  \\    \toprule
\multirow{7}{*}{$q=0$} &\multirow{3}{*}{CS}
&{PDASC}& \cite{jiao2015primal} &$--$&\textcircled{1}\textcircled{3}  &Finite &\textcircled{1}\textcircled{3}\\
&&{SDAR}& \cite{huang2018constructive} &$--$& \textcircled{2}\textcircled{3} &Finite &\textcircled{2}\textcircled{3}\\
&&{PSNP},{PCSNP}& Our & No & \textcircled{12}  &Finite & \textcircled{12} \\ \cmidrule{2-8}

&\multirow{4}{*}{General}&  {NL0R}& \cite{zhou2021newton}&\textcircled{5}  & \textcircled{5}\textcircled{7}  & Quadratic  & \textcircled{5}\textcircled{7}\textcircled{11}\\
&&{PSNP},{PCSNP}& {Our} &\textcircled{4}\textcircled{8}   &  \textcircled{4}\textcircled{8}\textcircled{12}  &Superlinear  & \textcircled{4}\textcircled{9}\textcircled{12} \\
&&{PSNP},{PCSNP}&{Our} &\textcircled{4}\textcircled{8}   &  \textcircled{4}\textcircled{8}\textcircled{12}  &Quadratic  & \textcircled{4}\textcircled{10}\textcircled{12} \\
&&{PSNP},{PCSNP}& Our  & \textcircled{5}  &  \textcircled{5}\textcircled{12}  &Quadratic & \textcircled{5}\textcircled{11}\textcircled{12} \\ \midrule

\multirow{7}{*}{$q\in(0,1)$} &\multirow{3}{*}{CS}&{IRucLq}&\cite{lai2013improved}& \textcircled{1}& $--$ &Superlinear  & \textcircled{1}\\
&&{PGA}&\cite{hu2021linear} & $--$& $--$  &Linear & \textcircled{18}\\
&&{PSNP},{PCSNP}& Our &No&   \textcircled{12}  &Quadratic & \textcircled{12} \\ \cmidrule{2-8} 

&\multirow{4}{*}{General}&  {IRL1}& \cite{lu2014iterative}&\textcircled{5} & $--$  & $--$ & $--$ \\
&&  {HpgSRN}& \cite{wu2022globally}&\textcircled{6}\textcircled{13}   & \textcircled{6}\textcircled{13}\textcircled{14}\textcircled{15}   &Superlinear  & \textcircled{6}\textcircled{13}\textcircled{14}\textcircled{15}\textcircled{16}\textcircled{17}\\
&&{PSNP},{PCSNP}&  {Our} & {\textcircled{8}} &  {\textcircled{8}\textcircled{12}} & Superlinear  & \textcircled{9}\textcircled{12} \\
&&{PSNP},{PCSNP}&  Our & {\textcircled{8}} &  {\textcircled{8}\textcircled{12}}  & Quadratic  & \textcircled{10}\textcircled{12} \\\botrule
 		\end{tabular}
	\footnotetext{Assumptions: 
	\textcircled{1} RIP (restricted isometry property), 
	\textcircled{2} SRC (sparse Rieze condition), \textcircled{3} Conditions on solutions, 
	\textcircled{4} $f$ is coercive when $q=0$,
	\textcircled{5} $f$ is strongly smooth,  
	\textcircled{6} $f$ is twice continuously differentiable on $\R^n$, 
	\textcircled{7} $f$ is locally strongly convex at $\bx^*$,
	\textcircled{8} $\nabla f$ is locally Lipschitz on a bounded set,
	\textcircled{9} $\nabla f$ is semismooth on a bounded set,
	\textcircled{10} $\nabla f$ is strongly semismooth on a bounded set,  
	\textcircled{11} $\nabla^2 f$ is locally Lipschitz at $\bx^*$,
	\textcircled{12} Second-order sufficient condition \eqref{2nd-sufficient-cond} holds at $\bx^*$, 
	\textcircled{13} $\nabla^2 f$ is locally Lipschitz on $\R^n$,
	\textcircled{14} $F$ satisfies K$\L{}$ property, 
	\textcircled{15} The sequence meets an angular condition, 
	\textcircled{16} A local error bound condition, 
	\textcircled{17} Second-order necessary condition \eqref{2nd-neccssary-cond} holds at $\bx^*$, 
	\textcircled{18} The sequence converges to a local minimizer.}
	\footnotetext{Abbreviations: `Global conv.' stands for that any accumulating point of the sequence is a certain minimum or stationary point.   `Sequence cov.' represents the whole sequence convergence. `$--$' means no such results. `No' signifies no need for additional conditions. `Finite'  means the sequence terminates within finitely many steps, a better convergence property than the quadratic rate.  } 
\end{table}  

\subsection{Contributions}
To the best of our knowledge, most existing work on this optimization necessitates conditions such as the twice continuous differentiability or strong smoothness of $f$, or the Lipschitz continuity of gradient $\nabla f$. Limited research has been dedicated to addressing this problem in cases where $\nabla f$ is locally Lipschitz or semismooth.  Therefore,  we embark on research within this trajectory, study the $L_q$ norm regularized optimization in a unified way for any $q\in[0,1)$, and develop an efficient numerical algorithm. The main contributions are threefold.
\begin{itemize} 
\item[A.] \textit{Optimality conditions under mild assumptions.} Based on the proximal operator of the $L_q$ norm with $q\in[0,1)$, we introduce the concept of a P-stationary point (refer to Definition \ref{p-stationary-point}) of \eqref{rso}. This allows us to build the first-order necessary as well as the second-order necessary and sufficient optimality conditions, which enable us to reveal the relationship between a P-stationary point and a local minimizer: a local minimizer is a P-stationary point if $\nabla f$ is locally Lipschitz. If we further assume the second-order sufficient condition,  a P-stationary point is a unique local minimizer.

\item[B.] \textit{Two novel algorithms having sequence and quadratic convergence.} We integrate the proximal operator and the Newton method to design two algorithms. They are proximal semismooth Newton pursuit ({PSNP}) and proximal conditional semismooth Newton pursuit ({PCSNP}). Both of them consist of two steps in each iteration. Firstly, we exploit the proximal operator of the $L_q$ norm to render a point $\bw^k$ with a reduced objective function value. Subsequently,  the support set of $\bw^k$ determines a subspace on which a semismooth Newton step is performed. Therefore, both algorithms are capable of ensuring a descent property while enabling rapid convergence, allowing us to prove the whole sequence convergence to a local minimizer at a superlinear or quadratic rate. The only difference between {PSNP} and {PCSNP} is whether to perform the Newton step when a condition is met, as outlined in Algorithms \ref{algorithm 1} and \ref{algorithm 2}.  

 We summarize the convergence outcomes for different algorithms in Table  \ref{tab:com-algs}. Notably, our algorithms {PSNP} and {PCSNP} consistently exhibit the global and sequence convergence under some more relaxed conditions. Specifically, for CS problems, we achieve the best results under $\textcircled{12}$ which is weaker than $\textcircled{1}$ and  $\textcircled{2}$. For the general problem with $q=0$,   {PSNP}, {PCSNP} and {NL0R} demonstrate the quadratic convergence rate but {PSNP} and {PCSNP} benefits from milder conditions due to $\textcircled{7} \Rightarrow \textcircled{12}$. For the general problem with $q\in(0,1)$,  {HpgSRN} achieves the accumulation point convergence under  $\textcircled{6}$, a stronger condition than $\textcircled{8}$ employed by {PSNP} and {PCSNP}. To establish the sequence convergence, it imposes condition $\textcircled{15}$ reliant on the sequence. Moreover, as shown in \cite[Theorem 4.9]{wu2022globally}, its most favorable convergence rate is superlinear with order $3/2$, whereas {PSNP} and {PCSNP} can converge quadratically. To the best of our knowledge, this work marks the first achievement of a quadratic convergence rate of an algorithm for solving  \eqref{rso} for $q\in(0,1)$.

\item[C.] \textit{High numerical performance.} We conduct some numerical experiments and compare {PSNP} and {PCSNP}  with several algorithms proposed for solving the $L_q$ norm regularized optimization. The results demonstrate that {PSNP}  and {PCSNP} are capable of delivering desirable sparse solutions in terms of accuracy as well as running relatively fast.
\end{itemize}

\subsection{Organization}
 The paper is organized as follows. In the next section, we introduce the notation to be used throughout this paper and explore properties of the proximal operator of the $L_q$ norm with $q\in[0,1)$. Section \ref{sec:opt-ana} is dedicated to establishing first and second-order necessary and sufficient optimality conditions for problem \eqref{rso}. We proceed to develop two proximal semismooth Newton pursuit algorithms and analyze their convergence properties in Section \ref{sec:pdnp}. This section also features the application of the derived theory to specific examples. Moving on to Section \ref{sec:num}, we conduct numerical experiments to demonstrate the performance of our proposed algorithms in comparison with several existing solvers. The paper concludes with a summary in the final section. 

\section{Preliminaries} 
In this section, we first introduce all notation and some useful functions, followed by giving relevant properties of these functions and the proximal operator of the $L_q$ norm.
\subsection{Notation} We first present all notation used in this paper.  The sign function is written as $\sign(t)$ which returns $0$ if $t=0$, $1$ if $t>0$, and $-1$ otherwise. The Euclidean norm for vectors and the spectral norm for matrices are represented as $\|\cdot\|$, and the infinity norm is denoted as $\|\cdot\|_{\infty}$.   We define $\supp(\bx)$ as the support set of vector $\bx$, containing all indices of its non-zero entries.  For an index set $\CS$, we use $\overline{\CS}$ to denote its complement, and $|\CS|$ signifies its cardinality   Given a vector $\bx$, we denote its neighborhood with a positive radius $\epsilon$ as ${\mathbb N}(\bx,\epsilon):=\{\bz \in\R^n: \|\bx-\bz \|^2< \epsilon\}$, and its sub-vector indexed by $\CS$ as $\bx_{\CS}\in\R^{|\CS|}$.  Furthermore, we represent the stacking of two vectors as $(\bx; \bz):=(\bx^\top \bz^\top)^\top$. Next, we introduce some concepts. \begin{itemize}
\item {\it (Locally) Lipschitz continuity:}  Function $g:\Omega \subseteq \R^n\rightarrow\R^m$ is  locally Lipschitz   on open set $\Omega$ with a positive constant $L$ if $\| g(\bw)- g(\bx)\|\leq L\| \bw-\bx\|$ for any $\bv,\bx\in \Omega$. In particular, if $\Omega = \mathbb{N}(\bx,\epsilon)$ for an $\epsilon>0$, we say $g$ is the locally Lipschitz at $\bx$. If $\Omega=\R^n$, we say $g$ is Lipschitz. 
\item{\it Strong smoothness:} $f$ is said to be $L$-strongly smooth on  $\R^n$ if for any $\bx,\bw\in\R^n$, 
\begin{eqnarray}\label{lip-L-def}
\begin{array}{ll}
f(\bx)\leq f(\bw)+\langle \nabla f(\bw), \bx-\bw \rangle +  \frac{L}{2} \|\bx-\bw\|^2.
\end{array}\end{eqnarray}
\item {\it Subdifferential:} For a proper and lower semi-continuous  function $h:\R^n\to\R$, its subdifferential $\partial h(\cdot)$ is well defined as \cite[Definition 8.3]{RW1998}. One can calculate that 
 \begin{eqnarray*}
\partial |z|^q = \left\{\begin{array}{ll}
\{q\sign(z)|z|^{q-1}\}, & |z|>0,\\
(-\infty, \infty), & |z|=0.
\end{array}\right.\end{eqnarray*}

\item {\it Generalized Hussain:}  Let $f$ be continuous differentiable and its gradient is locally Lipschitz on an open set $\Omega\subseteq\R^n$. Then by \cite{hiriart1984generalized}  the generalized Hussain matrix of $f$ at $\bx$, denoted by $\partial^2 f(\bx)$, is the set of matrices defined as, 
\begin{eqnarray*}\begin{array}{l}
 {\rm co}\left\{\H: \exists~ \bx^i\to\bx ~\text{with $f$ being twice  differentiable at $ \bx^i$   and~ $\nabla^2 f(\bx^i) \to \H$} \right\}.
 \end{array}\end{eqnarray*} 
where ${\rm co}\{\Omega\}$ stands for the convex hull of $\Omega$. It follows from \cite[Theorem 2.3]{hiriart1984generalized} that for any $\bx,\bw\in\Omega$, there is ${\bf z}\in{\rm co}\{\bx,\bw\}$ and $\H\in \partial^2 f({\bf z})$ such that
 \begin{eqnarray}\label{MVT}
\begin{array}{l}
 f(\bx)= f(\bw)+\langle \nabla f(\bw), \bx-\bw \rangle + \frac{1}{2} \langle \H (\bx-\bw), \bx-\bw \rangle .
 \end{array}\end{eqnarray} 
 \item {\it Semismoothness and directional derivative:} According to \cite{qi1993nonsmooth}, function $g:\Omega \subseteq \R^n\rightarrow\R^m$ is said to be semismooth at $\bx$ if $g$ is locally Lipschitz at $\bx$ and 
   \begin{eqnarray}\label{semi-s}
 \begin{array}{l}
  \lim_{\H\in\partial g(\bx+t\bd'),\bd'\to\bd,t\downarrow0} \{\H \bd'\}
 \end{array} \end{eqnarray} 
 exists for any $\bd\in\R^n$. 
The classic directional derivative of $g$ is   defined by
  \begin{eqnarray}\label{cdd}
 \begin{array}{ll}
 g'(\bx;\bd)~= \lim_{t \downarrow 0} \frac{g(\bx+t\bd)-g(\bx)}{t},\\
  g''(\bx;\bd)= \lim_{t \downarrow 0} \frac{\langle g(\bx+t\bd)-g(\bx), \bd\rangle }{t}=\langle g'(\bx;\bd),\bd \rangle.
 \end{array} \end{eqnarray} 
 The semismoothness implies that 
    \begin{eqnarray}\label{cdd-delta-semi}
 \begin{array}{l} 
 \H\bd  = g'(\bx;\bd) + o(\|\bd\|)
 \end{array} \end{eqnarray} 
for any $\H\in\partial g(\bx+\bd)$ and $\bd\to0$. In addition, for $\delta\in(0,1]$, $g$ is said to be strongly semismooth at $\bx$ if for any $\H\in\partial g(\bx+\bd)$ and $\bd\to0$ there is 
   \begin{eqnarray}\label{cdd-delta}
 \begin{array}{l} 
 \H\bd  = g'(\bx;\bd) + O(\|\bd\|^{2}).
 \end{array} \end{eqnarray} 
 It is easy to see that if $g$ is continuously differentiable at $\bx$, then $g$ is semismooth. If further $\nabla g$ is locally Lipschitz at $\bx$, then $g$ is strongly semismooth at $\bx$.
\end{itemize}
Throughout the paper, we say $f$ is LC$^1$ at $\bx$ (resp. on $\Omega$) if it is continuously differentiable and $\nabla f$ is locally Lipschitz at $\bx$ (resp. for any $\bx\in\Omega$). Similarly, we  say  $f$ is SC$^1$  at $\bx$ (resp. on $\Omega$) if $f$ is LC$^1$ and $\nabla f$ is   semismooth at $\bx$ (resp. for any $\bx\in\Omega$). Moreover,  $f$ is said to be SC$^2$ at $\bx$ (resp. on $\Omega$) if $f$ is LC$^1$ and $\nabla f$ is strongly semismooth at $\bx$ (resp. for any $\bx\in\Omega$). 
 One can observe that
\begin{eqnarray}\label{Lip-Smooth-continuous}
\begin{array}{rll}
\text{$f$ is twice continuously differentiable}&\Longrightarrow&\text{$f$ is $SC^1$ on an open set $\Omega\subseteq \R^n$}\\
&\Longrightarrow&\text{$f$ is $LC^1$ on an open set $\Omega\subseteq \R^n$}.
\end{array}
\end{eqnarray} 
Some properties associated with LC$^1$, SC$^1$, and SC$^2$ are given as follows.
 \begin{proposition}\label{pro-bd} Suppose $f$ is $LC^1$, $SC^1$,  or $SC^2$ at a non-zero point $\bx\in\R^n$ with $\CS:=\supp(\bx)$. Then $E(\cdot;\CS)$ is $LC^1$, $SC^1$,  or $SC^2$ at  $\bx$, respectively, where 
 \begin{eqnarray}\label{def-eta-c}
\begin{array}{l}
E(\cdot;{\cal T}):=f(\cdot)+\lambda \|(\cdot)_{{\cal T}}\|_q^q.
\end{array}
\end{eqnarray}
 Moreover,  there is a neighbourhood  $\mathbb{N}$ of $\bx$ and a constant $c\in(0,\infty)$ satisfying \begin{eqnarray}  
\label{bound-E-S}\sup\{\|\M\|:\M\in\partial^2 E(\bw;\CS), \forall \bw\in\mathbb{N} \}\leq c.\end{eqnarray} 
 \end{proposition}
 \begin{proof}
For any point $\bw$ in a neighbourhood  $\mathbb{N}$  of $\bx$, we must have $\CS\subseteq\supp(\bw)$. This implies that $\nabla \|(\cdot)_{\CS}\|_q^q$ is twice continuously differentiable on $\mathbb{N}$, and thus  $\nabla\|(\cdot)_{\CS}\|_q^q$ is locally Lipschitz, semismooth, and strongly semismooth on $\mathbb{N}$. As   $f$ being LC$^1$, SC$^1$, SC$^2$ at $\bx$,  $\nabla E(\cdot;\CS) = \nabla  f(\cdot)+\lambda\nabla\|(\cdot)_{\CS}\|_q^q$ is locally Lipschitz, semismooth, and strongly semismooth at $\bx$ \cite[Proposition 1.75.]{izmailov2014newton}. Finally, the locally Lipschitz continuity of $\nabla E(\cdot;\CS)$ on $\mathbb{N}$ and \cite[Sect. 2 ($P_1$)]{hiriart1984generalized} enable \eqref{bound-E-S}  immediately. 
 \end{proof}
 Hereafter, for an index set $\CS$, we write the sub-gradient  as $\nabla_{\CS}  f(\bx):=(\nabla f(\bx))_{\CS}$ and $\partial^2_{\CS}  f(\bx):=\{\H_{\CS}: \H\in \partial^2  f(\bx) \}$, where $\H_{\CS}$ represents  the principal sub-matrix containing rows and columns indexed by ${\CS}$.  Moreover,   we denote
 \begin{eqnarray*}
 \nabla_{\CS} E( \bx):=\nabla_{\CS} E( \bx;{\CS}),\qquad\partial^2_{\CS} E( \bx):=\partial^2_{\CS} E( \bx; \CS).
 \end{eqnarray*}
 Finally, for a non-zero point $\bx\in\R^n$ with $\CS:=\supp(\bx)$, if $\nabla E( \cdot;\CS)$ is $SC^1$ at $\bx$, then $(\nabla E)'(\bx;\CS;\bd)$ exists for any $\bd$ at $\bx$ \cite{qi1993nonsmooth}. In such a case, we denote $$ (\nabla E)'_{\CS}(\bx;\bd):= (  (\nabla E)'(\bx;\CS;\bd))_{\CS}.$$
  \begin{proposition}\label{pro-bd-expansion} Suppose $f$ is $SC^1$ at a non-zero point $\bx\in\R^n$ with $\CS:=\supp(\bx)$. Then for any $\M\in\partial^2_{\CS} E(\bx)$  any $\bd\in\R^{n}\to 0$ with $\bd_{\overline\CS}=0$, 
 \begin{eqnarray} \label{Md-Ed}
\begin{array}{r}
\M\bd_{\CS} - (\nabla E)'_{\CS}(\bx;\bd)=o(\|\bd\|),\\[1ex]
\nabla_{\CS} E( \bx +\bd)-\nabla_{\CS} E( \bx )-(\nabla E)'_{\CS}(\bx;\bd)=o(\|\bd\|).
\end{array}
\end{eqnarray} 
If $f$ is $SC^2$ at $\bx$, then
 \begin{eqnarray} \label{Md-Ed-2}
\begin{array}{r}
\M\bd_{\CS} - (\nabla E)'_{\CS}(\bx;\bd)=O(\|\bd\|^2),\\[1ex]
\nabla_{\CS} E( \bx +\bd)-\nabla_{\CS} E( \bx )-(\nabla E)'_{\CS}(\bx;\bd)=O(\|\bd\|^2).
\end{array}
\end{eqnarray} 
 \end{proposition}
 \begin{proof} By denoting $g(\cdot):=\nabla E( \cdot;\CS)$, Proposition \ref{pro-bd}  means that $g(\cdot)$ is  semismooth and strong semismooth at  $\bx\in\R^n$ if $f$ is $SC^1$ and $SC^2$, respectively, which together with 
 \eqref{cdd-delta-semi} and \eqref{cdd-delta} respectively leads to
  \begin{eqnarray*} 
 &&{\bf Q}\bd  = g'(\bx-\bd;\bd) + o(\|\bd\|), \\
 && {\bf Q}\bd  = g'(\bx-\bd;\bd) + O(\|\bd\|^2), 
 \end{eqnarray*} 
for any $ {\bf Q}\in \partial g(\bx-\bd+\bd )=\partial g(\bx)= \partial^2 E(\bx;\CS)$. Let $\M={\bf Q}_{\CS}\in \partial^2_{\CS} E(\bx)$. Moreover,  the semismoothness and strong semismoothness of $g(\cdot)$ at $\bx\in\R^n$ respectively implies  $g'(\bx-\bd;\bd)-g'(\bx;\bd)=o(\|\bd\|)$ from \cite[Theorem 2.3 (v)]{qi1993nonsmooth} and $g'(\bx-\bd;\bd)-g'(\bx;\bd)=O(\|\bd\|^2)$  from \cite[Lemma 2.2]{qi2003strongly}, thereby respectively  resulting in
   \begin{eqnarray*} 
&&{\bf Q}\bd - g'(\bx;\bd)= g'(\bx-\bd;\bd) + o(\|\bd\|)- g'(\bx;\bd)= o(\|\bd\|), \\
&&{\bf Q}\bd - g'(\bx;\bd)= g'(\bx-\bd;\bd) + O(\|\bd\|^2)- g'(\bx;\bd)= O(\|\bd\|^2), 
\end{eqnarray*} 
which by $({\bf Q}\bd - g'(\bx;\bd))_{\CS}={\bf Q}_{\CS}\bd_{\CS} -  (g'(\bx;\bd))_{\CS}= \M\bd_{\CS} - ((\nabla E)'(\bx;\CS;\bd)) _{\CS}= \M\bd_{\CS} - (\nabla E)' _{\CS}(\bx;\bd)$  yields the first equations in \eqref{Md-Ed}  and   \eqref{Md-Ed-2}, where the second ones  follow  from \cite[(2.17)]{qi1993nonsmooth}.
 \end{proof}
\subsection{Proximal Operator of the $L_q$ Norm} For a function $r:\R^n\to\R$, its proximal operator is defined by 
\begin{eqnarray*} 
\begin{array}{l}
{\rm Prox}_{r(\cdot)}(\bx):= {\rm argmin}_{\bw\in\R^n} ~ \frac{1}{2}\|\bx-\bz\|^2+r(\bw).  
\end{array}
\end{eqnarray*}
   For given  $\lambda>0$ and $q\in[0,1)$, we define two useful constants in the paper by
\begin{eqnarray} \label{prox-a-ct}
 \begin{array}{rrl}
c(\lambda,q):=(2\lambda (1-q))^{\frac{1}{2-q}}>0,~~~
\kappa(\lambda,q): =(2-q)\lambda^{\frac{1}{2-q}} (2 (1-q))^{\frac{1-q}{q-2}} .
\end{array}
\end{eqnarray}  
We present some properties of the proximal operator of $|\cdot|^q$ as follows.
\begin{proposition}\label{proximal-phi} For given scalar $a\in\R$, $\lambda>0$ and $q\in[0,1)$,  
\begin{eqnarray} \label{prox-a-rt}
\begin{array}{lll}
 {\rm Prox}_{\lambda |\cdot|^q} (a) &=&{\rm argmin}_z ~\frac{1}{2}(z-a)^2+\lambda |z|^q=:\varphi(z)\\
 &=& \left\{\begin{array}{lll}
\{0\}, &|a|< \kappa(\lambda,q),\\
\left\{0, \sign(a) c(\lambda,q) \right\},& |a|=  \kappa(\lambda,q),\\
\left\{\sign(a)\varpi_q (|a|)\right\}, &|a|> \kappa(\lambda,q),
\end{array} \right.
\end{array}\end{eqnarray}  
 where $\varpi_ q (a) \in \{z:  z-a+\lambda  q \sign(z) z^{q-1} = 0, z>0\}$ is unique. 
Moreover, if   ${\rm Prox}_{\lambda |\cdot|^q} (a)$ contains a non-zero point $z^*$, then
 \begin{eqnarray} \label{prox-lq-loca-strong-convex}
 \begin{array}{lll}
|z^*|\geq c(\lambda,q),~~~\varphi''(|z^*|)\geq 1- \frac{q}{2}> \frac{1}{2}.  
\end{array}
\end{eqnarray}   
\end{proposition}
\begin{proof}  If $q=0$, then it is easy to compute that $\kappa(\lambda,0)=\sqrt{2\lambda}, ~c(\lambda,0)=\sqrt{2\lambda}$,  and $\varpi_ 0(a)=a\in \{z: z-a = 0\}$. These by \cite{attouch2013convergence, zhou2021newton} can show that ${\rm Prox}_{\lambda  |\cdot|_0^0}$ takes the form of \eqref{prox-a-rt}.  Next, we focus on $q\in(0,1)$. The result is obvious when $a=0$. When $a>0$, the result follows from \cite[Proposition 2.4]{chen2014global}.  When  $a<0$, it has $ {\rm Prox}_{\lambda |\cdot|^q} (-a)= -  {\rm Prox}_{\lambda |\cdot|^q} (a)$ and we can consider $-a>0$, which also leads to the conclusion.

  The lower bound in \eqref{prox-lq-loca-strong-convex} for the magnitudes of non-zero entries has been established in many publications, such as \cite[Corollary 2.1]{ito2013variational}, \cite[Theorem 2.1]{chen2010lower}, \cite[Theorem 2.3]{lu2014iterative}, \cite[Lemma 2.1]{wu2022globally}. We omit the proof for simplicity. 
Finally,  $\varphi''(|z^*|)=1- \lambda q (1-q) |z^*|^{q-2} \geq 1- q/2  > 1/2$ due to $|z^*|\geq c(\lambda,q)$, finishing the proof.
\end{proof}

We  point  out that the proximal operator in \eqref{prox-a-rt} admits closed forms when $q=0$ by \eqref{prox-a-rt}, $q=1/2$ \cite{xu2012,chen2014global}, and $q=2/3$ \cite{krishnan2009fast, cao2013fast, chen2016computing}, which are given as follows.
\begin{itemize}
\item When $q=0$, it follows from \eqref{prox-a-rt} that
\begin{eqnarray}\label{proximal-l-0}
 \eqspace{1}
 {\rm Prox}_{\lambda  |\cdot|^0} (a)=\left\{\begin{array}{lll}
\{0\}, &|a|<\sqrt{2\lambda},\\
\{0,a\},~~& |a|=\sqrt{2\lambda},\\
\{a\}, &|a|>\sqrt{2\lambda}.
\end{array} \right.
\end{eqnarray}
\item When $q={1}/{2}$, by letting  $\phi:={\rm arcos}(\frac{\lambda}{4}(\frac{|a|}{3})^{-3/2})$, it follows from \cite{xu2012,chen2014global} that
\begin{eqnarray}\label{proximal-l-1/2}
\eqspace{1.5}
 {\rm Prox}_{\lambda  |\cdot|^{1/2}} (a)=\left\{\begin{array}{lll}
\{0\}, &|a|<  {3\lambda^{2/3}}/{2},\\
\left\{0,\sign(a)\lambda^{2/3}\right\},& |a|= {3\lambda^{2/3}}/{2},\\
\Big\{\frac{4a}{3} \cos^2\Big(\frac{\pi-\phi}{3}\Big)\Big\},~~ &|a|> {3\lambda^{2/3}}/{2}.
\end{array} \right.
\end{eqnarray} 
\item When $q={2}/{3}$, by letting
\begin{eqnarray*}
 \eqspace{1}
 \begin{array}{lll}\psi:= \left( \frac{a^2}{2} + \sqrt{\frac{a^4}{4} - \left(\frac{8\lambda}{9}\right)^3 }\right)^{1/3}+\left( \frac{a^2}{2} - \sqrt{\frac{a^4}{4} - \left(\frac{8\lambda}{9}\right)^3 }\right)^{1/3},\end{array} 
\end{eqnarray*} 
it follows from \cite{krishnan2009fast, cao2013fast, chen2016computing} that
\begin{eqnarray}\label{proximal-l-2/3}
\eqspace{1.5}
 {\rm Prox}_{\lambda  |\cdot|^{2/3}} (a)=\left\{\begin{array}{lll}
\{0\}, &|a|<  2(2\lambda/3)^{3/4},\\
\left\{0,\sign(a)(2\lambda/3)^{3/4}\right\},& |a|=  2(2\lambda/3)^{3/4},\\
\Big\{\frac{\sign(a)}{8}\Big( \sqrt{\psi}+\sqrt{2|a|/\sqrt{\psi}-\psi}\Big)^3\Big\},~~ &|a|>  2(2\lambda/3)^{3/4}.
\end{array} \right.
\end{eqnarray}
\end{itemize}
 For other choices $q\in(0,1)$, it can be computed efficiently by an algorithm proposed in \cite{chen2016computing} or the Newton method mentioned in \cite{peng2018global} (which will be developed in Section \ref{sec:pdnp}, see Algorithm \ref{algorithm-proxmalq}). Moreover, \eqref{prox-lq-loca-strong-convex} implies that if proximal operator ${\rm Prox}_{\lambda  |\cdot|^q} (a)$  contains a non-zero element, then the module of the element has a positive lower bound $c(\lambda,q)$, which only depends on $\lambda, q$ and is irrelevant to $a$.  This property is beneficial to establish the convergence results in Section \ref{sec:global-convergence}.
\section{Optimality Analysis}\label{sec:opt-ana}
Based on the proximal operator of $|\cdot|^q$, we can calculate the one of the $L_q$ norm by
 \begin{eqnarray*}  
 \begin{array}{lll}
{\rm Prox}_{\lambda\|\cdot\|_q^q}(\bx)=\left\{\bz\in\R^n: ~w_i\in {\rm Prox}_{\lambda|\cdot|^q}(x_i),~ i=1,2,\ldots,n\right\}, 
\end{array}
\end{eqnarray*}
which allows us to define a P-stationary point for \eqref{rso} as follows. 
\begin{definition}[{\bf P-stationary point}]\label{p-stationary-point}
A point $\bx^*$ is called a P-stationary point  of   problem (\ref{rso}) if there exists an $\alpha>0$ such that
\begin{eqnarray}\label{alpha-stationary-point-0}
\bx^{*} \in \prox (\bx^{*}-\alpha \nabla f(\bx^{*}) )
\end{eqnarray}
\end{definition}
Based on the above definition, we have
\begin{eqnarray}\label{alpha-stationary-point}
 \begin{array}{lll}
\bx^{*} \in   {\rm argmin}_{\bx}~ J(\bx;\bx^*):= \|\bx-(\bx^{*}-\alpha \nabla f(\bx^{*}) )\|^2+2\alpha\lambda\|\bx\|^q_q.
\end{array} 
\end{eqnarray} 
Since solution $0$ to problem (\ref{rso}) is trivial and not the objective we are aiming for, we need to exclude it from the set of P-stationary points.  This can be accomplished by selecting $\lambda$ to be smaller than a certain threshold as follows.
\begin{lemma}\label{exclude-0} Suppose   $\nabla f(0)\neq 0$. Then for a given $\alpha>0$ and for any $\lambda$ satisfying
\begin{eqnarray} \label{bound-lambda}
 \eqspace{1}
 \begin{array}{lll}
0< \lambda < \overline{\lambda}:= \frac{\alpha^{1-q} }{2}\|\nabla f(0)\|_{\infty}^{2-q},
\end{array}
\end{eqnarray}
  point $0$ is not a P-stationary point with  $\alpha>0$.
\end{lemma}
\begin{proof} Suppose  $0$ is a P-stationary point with  $\alpha>0$. Then it has 
\begin{eqnarray*} 
 \begin{array}{lll}
0 \in   {\rm argmin}~  \|\bx+ \alpha \nabla f(0) \|^2+2\alpha\lambda\|\bx\|^q_q.
\end{array} 
\end{eqnarray*}
 Now we consider a point $\bx^*$ with $x^*_{s }=-\alpha \nabla_{s } f(0) $ and $\bx^*_i=0$ for any $i\neq s $, where $s $ is one of the indices satisfying $ |\nabla_{s } f(0)| = \|\nabla f(0)\|_{\infty}$, 
which together with $\overline{\lambda}\geq \lambda$ in condition \eqref{bound-lambda} results in $2\alpha\lambda|\alpha\nabla_s f(0)|^q <  {\alpha^2 |\nabla_{s } f(0)|^2}$. This allows us to derive that
\begin{eqnarray*} 
&& \|\bx^*+ \alpha \nabla f(0) \|^2+ 2\alpha\lambda\|\bx^*\|^q_q\\
&\geq&   \|0+ \alpha \nabla f(0) \|^2+2\alpha\lambda \|0\|^q_q\\
&=&  \alpha^2 \begin{array}{l}
{\sum}_{i\neq s }
\end{array}  | \nabla_i f(0) |^2 + \alpha^2 | \nabla_{s } f(0) |^2\\
&>&  \alpha^2  \begin{array}{l}
{\sum}_{i\neq s }
\end{array} | \nabla_i f(0) |^2 +  \alpha^2 (x^*_{s } + \alpha \nabla_{s } f(0))^2 + 2\alpha\lambda|\alpha \nabla_{s } f(0)|^q\\
&=& \|\bx^*+ \alpha \nabla f(0) \|^2+ 2\alpha\lambda\|\bx^*\|^q_q.
\end{eqnarray*}
The above contradiction means that $0$ is not a P-stationary point with  $\alpha>0$.
\end{proof}
The assumption, $\nabla f(0)\neq 0$, is reasonable. If $\nabla f(0) = 0$, then $0$ is a stationary point in terms of $0\in \partial F(0)$, resulting in the triviality of solving problem \eqref{rso}.  Therefore,   we will always assume that $\nabla f(0)\neq 0$ and condition \eqref{bound-lambda} to prevent $0$ from being a  P-stationery point in the sequel if no additional explanations are specified. Moreover, $\overline{\lambda}$ in \eqref{bound-lambda} can be used as an upper bound for $\lambda$ in the numerical experiments. 

\subsection{Some properties of P-stationary points} 
Hereafter for a given point $\bx^*$,  we always denote its support set as
$$\CS^*:=\supp(\bx^*).$$
We note that $\CS^*\neq \emptyset$ if $\bx^*$  is a P-stationary point as it is not $0$ from Lemma \ref{exclude-0}. 
\begin{lemma}\label{lemma-grad-gamma-0} A local minimizer $\bx^*$  of (\ref{rso}) satisfies   $\nabla_{\CS^*} F(\bx^*)=0$ and a P-stationary point  $\bx^*$  with an $\alpha>0$  satisfies  
\begin{eqnarray}\label{char-P-sta}
\begin{array}{l}
\nabla_{\CS^*} F(\bx^*)=0,~~ \|\nabla_{\overline\CS^*} f(\bx^*)\|_{\infty}\leq \frac{\kappa(\alpha\lambda,q)}{\alpha}.
\end{array}
\end{eqnarray}
\end{lemma}
\begin{proof} Let $\bx^*$  be a local minimizer. 
For $q=0$,  the result follows from \cite[Theorem 2.1]{Beck13}. For $q\in(0,1)$,  \cite[Theorem 10.1]{RW1998} leads to $0\in \partial F(\bx^*)$, delivering $\nabla_{\CS^*} F(\bx^*)=0$. Let $\bx^*$  be a P-stationary point. By  \eqref{alpha-stationary-point} and \cite[Theorem 10.1]{RW1998}, we have
$$0\in \bx^{*}-(\bx^{*}-\alpha \nabla f(\bx^{*}) )+\alpha\lambda \partial \|\bx^{*}\|^q_q=\alpha \partial F(\bx^*),$$ 
contributing to $\nabla_{\CS^*} F(\bx^*)=0$. Again, by \eqref{alpha-stationary-point} and \eqref{prox-a-rt}, for any $i\notin \CS^*$, we have $| \alpha \nabla_i f(\bx^{*})| =  |(\bx^{*}-\alpha \nabla f(\bx^{*})_i|\leq \kappa(\alpha\lambda,q),$ as desired.
\end{proof}
 \subsection{First-order optimality conditions}
 The subsequent theorem shows that a P-stationary point has a close relationship with a local minimizer of problem \eqref{rso}.
\begin{theorem}[{\bf First-order necessary condition}]\label{First-order-necessary-condition} Let  $\bx^*$  be a local minimizer of  (\ref{rso}). Suppose $f$ is $LC^1$ at $\bx^*$. Then there is an $\epsilon_*>0$ such that $\bx^*$ is a $P$-stationary point with $0<\alpha<\alpha_*:=\min\{\alpha_1^*,\alpha_2^*\}$, where 
\begin{eqnarray*}
\begin{array}{lll}
\alpha_1^*:=  \frac{\epsilon_*}{2\lambda\|\bx^*\|_q^q+2\sqrt{\lambda^2\|\bx^*\|_q^{2q}+\epsilon_* \|\nabla f(\bx^*)\|^2}},~~\alpha_2^*:=\frac{1}{\sup_{\bx\in {\mathbb N}( \bx^*, \epsilon_*)} \sup_{\H\in\partial^2 f(\bx)}\|\H\|}. 
\end{array} 
\end{eqnarray*}  
\end{theorem}
\begin{proof} As $\bx^*$  is a local minimizer, there is a $\varepsilon_*>0$ satisfying $\bx^*\in{\rm argmin}\{F(\bx):\bx\in\mathbb{N}(\bx^*,\varepsilon_*)\}$. Since $f$ is $LC^1$ at $\bx^*$,   there is an $\epsilon_*\in(0,\varepsilon_*)$ such that  $\nabla f(\cdot)$ is Lipschitz on $\mathbb{N}(\bx^*,\epsilon_*)$, resulting the boundedness of  $\partial^2 f(\cdot)$ on ${\mathbb N}( \bx^*, \epsilon_*)$ and consequently $\alpha_2^* > 0$. This means $\alpha\in(0,\alpha_*)$ well defined.   Suppose $\bx^*$ is not a $P$-stationary point with such an $\alpha$. Consider any point satisfying $\bz^{*}  \in \prox (\bx^{*}-\alpha \nabla f(\bx^{*}) )=:{\mathbb P}$ and denote $\bd^*:=\bz^*-\bx^*$.
This indicates
\begin{eqnarray*}
 \eqspace{1}
\begin{array}{lll}
\frac{1}{2}\|\bz^*-(\bx^{*}-\alpha \nabla f(\bx^{*}) )\|^2+\alpha\lambda\|\bz^*\|^q_q \leq\frac{1}{2}\|\alpha \nabla f(\bx^{*}) \|^2+\alpha\lambda\|\bx^*\|^q_q.
\end{array} 
\end{eqnarray*} 
After simple manipulations, the above inequality leads to
\begin{eqnarray}\label{z-x-alpha}
\begin{array}{lll}
&& \langle \bd^*,  \nabla f(\bx^{*}) \rangle + \lambda \|\bz^*\|^q_q - \lambda \|\bx^*\|^q_q \leq -\frac{1}{2\alpha}\|\bd^*\|^2, 
\end{array} 
\end{eqnarray} 
which by $\langle \bd^*,  \nabla f(\bx^{*}) \rangle  \geq -\|\bd^*\|^2/(4\alpha)- \alpha \|\nabla f(\bx^{*}) \|^2 $ and $\alpha\leq\alpha_*\leq\alpha_1^*$ yields
   \begin{eqnarray*} \eqspace{1}
 \begin{array}{lll}   \| \bd^*\|^2 &\leq &   4\alpha^2 \|\nabla f(\bx^*)\|^2 +   4\alpha  \lambda \|\bx^*\|^q_q   < \epsilon_*.
 \end{array}  \end{eqnarray*} 
Hence $\bw^*\in  \mathbb{N}(\bx^*,\epsilon_*)$ and so is $\bfz^*$ for any $\bfz^*\in \co\{\bx^*,\bw^*\}$. As a result,  $\|\H^*\|\leq 1/ \alpha_2^*\leq 1/ \alpha_*$ for any $\H^*\in\partial ^2 f(\bfz^{*})$, which by \eqref{MVT} allows us to derive that
\begin{eqnarray*} 
0&\leq& 2f(\bz^*) + 2\lambda\|\bz^*\|^q_q  - 2f(\bx^*) - 2\lambda\|\bx^*\|^q_q  \\ 
&=&  2\langle \bd^*,  \nabla f(\bx^{*}) \rangle  +   \langle \H^* \bd^*, \bd^* \rangle + 2\lambda \|\bz^*\|^q_q  - 2\lambda\|\bx^*\|^q_q  \\ 
&\leq& ( 1/\alpha_* -  {1}/\alpha  )\| \bd^*\|^2\leq 0,
\end{eqnarray*}
where the three inequalities are due to the local optimality of $\bx^*$ and  $\bw^*\in  \mathbb{N}(\bx^*,\epsilon_*)$,   \eqref{z-x-alpha}, and $0<\alpha<\alpha_*$, respectively,   showing the desired result.
\end{proof} 
\begin{remark}Similar results to the above theorem have been achieved in \cite{Beck13, zhou2021newton}. But it stated that any global minimizer is a P-stationary point if $\nabla f$ is Lipschitz on $\R^n$ \cite{Beck13}  or $f$ is strongly smooth on $\R^n$ \cite{zhou2021newton}. Here, our theorem shows that any local minimizer is a P-stationary point as long as $f$ is $LC^1$. Therefore, according to \eqref{Lip-Smooth-continuous}, we derive a stronger result but under a weaker condition. \end{remark}
 \subsection{Second-order optimality conditions}
Recalling the definition in \eqref{def-eta-c}, one can observe that if $\CS=\supp(\bx)$, then 
\begin{eqnarray}\label{partial-2-E}
\partial^2_{\CS} E( \bx )=\partial^2_{\CS} f(\bx)+\lambda \nabla^2 \|\bx_{\CS}\|^q_q
\end{eqnarray}
due to $\|(\cdot)_{\CS}\|^q_q$ is twice continuously differentiable at $\bx$. 
\begin{theorem}[{\bf Second-order necessary or sufficient  condition}] Suppose  $f$ is $LC^1$ at $\bx^*$. Then the following statements are valid. \label{second-order-necessary-condition} ~
\begin{itemize}
\item[1)] If $\bx^*$ is a local minimizer  of problem (\ref{rso}), then it satisfies
\begin{eqnarray}\label{2nd-neccssary-cond}
\begin{array}{l}
\langle \M^*\bv, \bv \rangle \geq  0,~~\exists~\M^*\in \partial^2_{\CS^*} E( \bx^*),~\forall~\bv\in\R^{|\CS^*|}.\qquad
\end{array}
\end{eqnarray}
\item[2)] If $\bx^*$ is a P-stationary point  of problem (\ref{rso}), then it is also a uniquely local minimizer if  it satisfies
\begin{eqnarray}\label{2nd-sufficient-cond}
\begin{array}{l}
\langle \M^*  \bv, \bv \rangle >  0,~~\forall~\M^*\in \partial^2_{\CS^*}  E( \bx^*),~\forall~\bv(\neq0)\in\R^{|\CS^*|}.
\end{array}
\end{eqnarray}
Namely, there is an $\eta_*{>}0$ and a neighbourhood $\N$ of $\bx^*$ satisfying
\begin{eqnarray}\label{2nd-sufficient-cond-growth}
\begin{array}{l}
 F(\bx ) \geq F(\bx^{*}) + \eta_*\|\bx-\bx^{*}\|^2,~~\forall ~\bx\in  \N.
\end{array}
\end{eqnarray}
\end{itemize}
\end{theorem}
\begin{proof} Before we prove the results, we claim some useful facts.  By Proposition \ref{pro-bd},
there exists a neighbourhood  $\mathbb{N}^*$ of $\bx^*$ and a positive constant $c_*$ satisfying 
\begin{eqnarray}  
\label{bound-E-ssS}\sup\{\|\M\|:\M\in\partial^2 E(\bx;\CS^*), \forall \bx\in\mathbb{N}^* \}\leq c_*.\end{eqnarray} 
It is evident that $\nabla_{\CS^*} E( \bx^*) = \nabla_{\CS^*}  F(\bx^{*})$ due to definition  \eqref{def-eta-c}. If $\bx^*$ is a locally optimal solution or a P-stationary point,  $\nabla_{\CS^*} F( \bx^*)   =0$ from Lemma \ref{lemma-grad-gamma-0}, thereby  
\begin{eqnarray}  
\label{E-S-0}\nabla_{\CS^*} E( \bx^*) =0.\end{eqnarray}  
1) For any $\bv\in\R^{|\CS^*|}$ and sufficiently small $\tau>0$,  by letting $\bd_{\CS^*}=\bv$ and $\bd_{\overline{\CS}^*}=0$, the local optimality enables us to derive \eqref{2nd-neccssary-cond} immediately due to
\begin{eqnarray*} 
0&\leq& 2F(\bx^*+\tau \bd) - 2F(\bx^*)
=2E( \bx^*+\tau \bd;\CS^*) - 2E( \bx^*;\CS^*)\\ 
&=&2\tau \langle \nabla_{\CS^*}  E( \bx^{*}), \bv \rangle +  \tau ^2 \langle \M^\tau \bv, \bv\rangle =  \tau ^2 \langle \M^\tau   \bv, \bv\rangle,
\end{eqnarray*} 
where $\M^\tau\in \partial^2_{\CS^*} E( \bx^{\tau})$ and $\bx^{\tau}\in {\rm co}\{\bx^*+\tau \bd,\bx^*\}$ and the second equality is from \eqref{MVT}, resulting in $\langle \M^\tau  \bv, \bv\rangle\geq 0$. Note that $\bx^\tau\in\mathbb{N}^*$ for sufficiently small $\tau$ and thus   sequence  $\{\M^\tau\}$ is bounded owing to \eqref{bound-E-ssS}, namely it has a subsequence converging to $\M^*$. It follows from  \cite[Sect. 2 ($P_2$)]{hiriart1984generalized}  that $\M^* \in \partial^2_{\CS^*} E( \bx^{*})$. Finally, by passing to the subsequence, taking the limit of $\langle \M^\tau \bv, \bv\rangle\geq 0$ along with $\tau\to0$ leads to $\langle \M^*  \bv, \bv\rangle\geq 0$, as desired.

2) Suppose the conclusion is not true. Then there exists a sequence $\{\bx^k=\bx^*+\bd/k\}$, where $\bd \in\R^n$ with $\bd_{\CS^*}\neq 0$, such that
 \begin{eqnarray} \label{k3-d}
\begin{array}{l}
(1/k^{3})\|\bd\|^2 = (1/k)\|\bx^k-\bx^*\|^2 > F(\bx^k)-F(\bx^*).
\end{array}
\end{eqnarray}
Since $f$ is LC$^1$ at $\bx^*$, $\nabla E(\cdot;\CS^*)$ is locally Lipschitz around $\bx^*$. Moreover, any point $\bfz^k\in\co\{\bx^k,\bx^*\}$ for sufficiently large $k$ is near $\bx^*$. Now by \eqref{MVT} we obtain
\begin{eqnarray} \label{gap-E-E-k}
2E( \bx^k;\CS^*) -   2E( \bx^*;\CS^*) &=& (1/k^2)\left\langle{\bf Q}^k \bd, \bd\right\rangle+ (2/k)\left\langle\nabla E(\bx^*), \bd\right\rangle \nonumber\\
&=& (1/k^2)\left\langle {\bf Q}^k \bd, \bd\right\rangle+ (2/k)\left\langle\nabla_{\overline{\CS}^*} E(\bx^*), \bd_{\overline{\CS}^*}\right\rangle  \nonumber\\
&=& (1/k^2)\left\langle {\bf Q}^k \bd, \bd\right\rangle+ (2/k)\left\langle\nabla_{\overline{\CS}^*} f(\bx^*), \bd_{\overline{\CS}^*}\right\rangle  \nonumber\\
&\geq&  (1/k^2)\left\langle {\bf Q}^k\bd, \bd\right\rangle-(t/ k)\| \bd_{\overline{\CS}^*}\|_1 ,
\end{eqnarray}
where ${\bf Q}^k  \in\partial^2 E(\bfz^k;\CS^*)$,  the second equality is from \eqref{E-S-0}, and the last inequality holds with $t:=2\kappa(\alpha\lambda,q)/\alpha $ by \eqref{char-P-sta}. Now, we decompose ${\bf Q}^k$ and define $\sigma_*$  as 
 \begin{eqnarray} \label{def-sigma-*-E}
~~~~{\bf Q}^k=\left[\begin{array}{cc}
\M^k ~~  & {\bf Q}^k_{12} \\
{\bf Q}^{k\top}_{12} &   {\bf Q}^k_{\overline{\CS}^*} 
\end{array}\right] ,~~~\sigma_* := \inf_{\M^*\in \partial^2_{\CS^*} E(\bx^*)}\inf_{\bv\in\R^{|\CS^*|}} \frac{\langle\M^*\bv, \bv \rangle}{2\|\bv\|^2}>0,
\end{eqnarray}
where $\M^k= {\bf Q}^k_{{\CS}^*}$ and ${\bf Q}^k_{12}$ is the upper diagonal block of ${\bf Q}^k$, and $\sigma_*>0$ due to \eqref{2nd-sufficient-cond}. Then we have
\begin{eqnarray} \label{dMd}
\left\langle {\bf Q}^k \bd, \bd\right\rangle&=&\left\langle\M^k \bd_{\CS^*}, \bd_{\CS^*}\right\rangle+2\left\langle \bd_{\CS^*}, {\bf Q}^k_{12} \bd_{\overline{\CS}^*}\right\rangle+\left\langle{\bf Q}^k_{\overline{\CS}^*}\bd_{\overline{\CS}^*}, \bd_{\overline{\CS}^*}\right\rangle\nonumber\\
&\geq&\left\langle\M^k\bd_{\CS^*}, \bd_{\CS^*}\right\rangle-\sigma_*\|\bd_{\CS^*}\|^2-( \|{\bf Q}^k_{12} \|^2/\sigma_*)\|\bd_{\overline{\CS}^*}\|^2-\|{\bf Q}^k _{\overline{\CS}^*}\|\|\bd_{\overline{\CS}^*}\|^2\nonumber\\
&\geq&\left\langle\M^k \bd_{\CS^*}, \bd_{\CS^*}\right\rangle-\sigma_*\|\bd_{\CS^*}\|^2-( c_*^2/\sigma_*+c_*)\|\bd_{\overline{\CS}^*}\|^2,
\end{eqnarray}
where the first inequality used a fact $2\langle\bx,\bw\rangle\leq ({1}/{\pi}) \|\bx\|^2+ \pi  \|\bw\|^2$ for any $\pi>0$ and the second one is due to \eqref{bound-E-ssS}. Recalling \eqref{k3-d}, we can obtain that
\begin{eqnarray*} 
(2/k^{3})\|\bd_{\CS^*}\|^2&>& 2F(\bx^k) - 2F(\bx^*)-(2/k^{3})\|\bd_{\overline{\CS}^*}\|^2 \\
&=& 2E( \bx^k;\CS^*)+2\lambda\|\bx^k_{\overline{\CS}^*}\|_q^q  -  2E( \bx^*;\CS^*)   -(2/k^{3})\|\bd_{\overline{\CS}^*}\|^2 \\ 
&\geq& (1/k^2)\left\langle{\bf Q}^k \bd, \bd\right\rangle-(t/k)\| \bd_{\overline{\CS}^*}\|_1 + (2\lambda/k^q)\|\bd_{\overline{\CS}^*}\|_q^q    -(2/k^{3})\|\bd_{\overline{\CS}^*}\|^2 \\
&\geq&  (2\lambda/k^q)\|\bd_{\overline{\CS}^*}\|_q^q -(t/k)\| \bd_{\overline{\CS}^*}\|_1 -(( c_*^2/\sigma_*+c_**)/k^2+2/k^{3})\|\bd_{\overline{\CS}^*}\|^2\\
&+& (1/k^2)\left\langle\M^k\bd_{\CS^*}, \bd_{\CS^*}\right\rangle-(\sigma_*/k^2)\|\bd_{\CS^*}\|^2\\
&\geq&   (1/k^2)\left\langle\M^k\bd_{\CS^*}, \bd_{\CS^*}\right\rangle-(\sigma_*/k^2)\|\bd_{\CS^*}\|^2,
\end{eqnarray*}
where the second inequality is due to $\|\bx^k_{\overline{\CS}^*}\|_q^q=\|\bd_{\overline{\CS}^*}/k\|_q^q$ and \eqref{gap-E-E-k}, the third inequality is from \eqref{dMd}, and the last one holds for sufficiently large $k$ as $q\in(0,1)$. The above condition immediately results in ${(2/{k})\|\bd_{\CS^*}\|^2}>  \left\langle\M^k\bd_{\CS^*}, \bd_{\CS^*}\right\rangle -\sigma_*\|\bd_{\CS^*}\|^2.$ 
By taking the limit (along a subsequence if necessary) of both sides of the above condition, condition \eqref{def-sigma-*-E} leads to $$0\geq  \frac{\left\langle\M^*\bd_{\CS^*}, \bd_{\CS^*}\right\rangle}{\|\bd_{\CS^*}\|^2}-\sigma_*\geq 2\sigma_*-\sigma_*>0.$$
This contradiction shows the result.
\end{proof}
To end this section, we pay attention to second-order sufficient condition \eqref{2nd-sufficient-cond} which can be guaranteed by the following condition. 
\begin{corollary}\label{suf-2nd-sufficient-cond} Let $\bx^*$ be a P-stationary point of  (\ref{rso}) with $\alpha>0$.  Suppose $f$ is $LC^1$ at $\bx^*$ and for any $\H^*\in \partial^2_{\CS^*} f(\bx^*)$ there is
\begin{eqnarray}\label{2nd-neccssary-cond-cs}
\begin{array}{l}
\lambda_{\min} :=\lambda_{\min} \left(\H^*  \right) > \frac{q}{2\alpha},
\end{array}\end{eqnarray}
where $\lambda_{\min}(\H^*)$ is the minimal eigenvalue of $\H^*$, then  (\ref{2nd-sufficient-cond}) holds.
\end{corollary}
\begin{proof} For $q=0$, $\partial^2_{\CS^*} E( \bx ^*)=\partial^2_{\CS^*} f( \bx^* )$ is positive definite due to $\lambda_{\min}>0$ hence condition (\ref{2nd-sufficient-cond}) holds. For $q\in(0,1)$, as $\bx^*$ is a solution to \eqref{alpha-stationary-point}, then $|x_i^*|\geq c(\alpha\lambda,q)$ by \eqref{prox-lq-loca-strong-convex} for any $i\in\CS^*$. Now one can verify that
\begin{eqnarray*} 
\lambda_{\min} -  \frac{\lambda q (1-q)}{\min _{i\in \CS^*} |x_i^*|^{2-q} } \geq 
\lambda_{\min} -  \frac{\lambda q (1-q)}{(c(\alpha\lambda,q) )^{2-q} } = \lambda_{\min} -  \frac{\lambda q (1-q)}{2\alpha\lambda (1-q)}  >0,
\end{eqnarray*}
where the equality and the last inequality are from  \eqref{prox-a-ct}  and \eqref{2nd-neccssary-cond-cs}. This condition and \eqref{partial-2-E} ensure the positive definiteness of any elements of $\partial^2_{\CS^*} E( \bx ^*)$, thereby guaranteeing condition (\ref{2nd-sufficient-cond}).
\end{proof}
\section{Proximal Semismooth Newton Pursuit} \label{sec:pdnp}
We begin by giving the Newton method to solve the proximal operator of $L_q$ norm, 
\begin{equation}
{\rm Prox}_{\lambda\|\cdot\|_q^q}(\bx) = {\rm argmin}_{\bw\in\R^n} ~ h(\bz;\bx):=\frac{1}{2}\|\bx-\bz\|^2+\lambda\|\bw\|_q^q.
\end{equation}
The scheme is presented in Algorithm \ref{algorithm-proxmalq}, where $\nabla^2 h(\bu^\ell;\bx_T)$ and $\nabla h(\bu^\ell;\bx_T)$ denote the Hessian and gradient of $h(\cdot;\bx_T)$ at $\bu^\ell$. As mentioned earlier, the proximal operator admits a closed-form solution when $q \in {0, 1/2, 2/3}$, as expressed by \eqref{proximal-l-0}, \eqref{proximal-l-1/2}, and \eqref{proximal-l-2/3}. For other choices of $q$, based on \eqref{prox-a-rt}, the proximal operator has non-zero entries only when $|x_i| > \kappa(\lambda, q)$. Therefore, we need to find entries $i \in T$ using the Newton method and set the remaining entries to $0$. According to numerical experiments, this process usually runs very quickly, requiring fewer than $5$ iterations to obtain the proximal operator for any $q \in [0, 1)$ accurately.
\begin{algorithm}[th] 
 \caption{ProxLq$(\bx,\lambda,q)$ }
\label{algorithm-proxmalq}
\begin{algorithmic}[1]
\If{$q\in\{0,1/2,2/3\}$}
\State Return ${\bf z}$ by \eqref{proximal-l-0}, \eqref{proximal-l-1/2}, and \eqref{proximal-l-2/3} accordingly.
\Else
\State Compute $T =\{i: |x_i|>\kappa(\lambda,q)\}$ using \eqref{prox-a-ct}, set ${\bf u}^{0}=\bx_T$ and $\ell = 0$. 
\While{$\|\nabla h({\bf u}^\ell;\bx_T)\|<10^{-8}$}
\State Compute $\bd^\ell =  ( \nabla^2 h({\bf u}^\ell;\bx_T))^{-1}\nabla h({\bf u}^\ell;\bx_T) $.
\State Find a finite integer $s_\ell\in\{0,1,2,\ldots\}$ such that 
\begin{eqnarray*} 
	 h(\bu^{\ell}-2^{-s_\ell}\bd^\ell;\bx_T)  \leq  h(\bu^{\ell};\bx_T) -  10^{-4} \|2^{-s_\ell}\bd^\ell\|^2, 
\end{eqnarray*}
 \State Update $\bu^{\ell+1}=\bu^{\ell}-2^{-s_\ell}\bd^\ell$ and $\ell=\ell+1.$
 \EndWhile
 \State Return ${\bf z}$ by ${\bf z}_T=\bu^{\ell+1}$ and ${\bf z}_{\overline T}=0$.
 
  \EndIf
\end{algorithmic}
\end{algorithm}

Now we are ready to develop the algorithm, PSNP, to solve problem \eqref{rso} based on the proximal operator of $L_q$ norm and the semismooth Newton method. More specifically,  the algorithmic frameworks consist of two major steps: 
\begin{itemize}
\item We first select a point $\bz^{k}$ from the proximal operator with a proper step size to reduce the objective function value in a desirable scale. That is, based on current point $\bx^k$, by applying Algorithm \ref{algorithm-proxmalq}, 
find an $\alpha_k$ such that 
   \begin{eqnarray}    
   \label{armijio-descent-property} 
&&\bz ^{k}={\rm ProxLq}(\bx^{k}- \alpha_k \nabla f(\bx^k),\alpha_k \lambda,q) \in  {\rm Prox}_{\alpha_k \lambda \|\cdot\|^q_q} (\bx^{k}- \alpha_k \nabla f(\bx^k)),\\[1ex]
\label{armijio-descent-property-1} 
&&F(\bz ^{k})\begin{array}{l}\leq F(\bx^k)-  \frac{\sigma}{4}\|\bz ^{k}-\bx^{k}\|^{2}.\end{array}
\end{eqnarray}  
\item Then a subspace  $\{\bv\in\R^{n}:\bv_{\overline{\CS}^k}=0\}$ is determined by support set $\CS^k$ of $\bz ^{k}$, where
\begin{eqnarray}\label{supp-wk-S}
\CS^k:={\rm supp}(\bw^{k}), 
\end{eqnarray}
 and a semismooth Newton step is performed on the subspace to find a direction $\bd^k$. That is, choose $\M^k\in\partial^2_{\CS^k} E(\bw^k)$ and calculate  $\bd^{k}$ by solving the following equations, 
\begin{eqnarray}\label{Newton-descent-property}
	  &&\M^k \bd^{k}_{\CS^k}  =   \nabla_{\CS^k}  E(\bz ^{k}), ~~~~\bd^{k}_{\overline{\CS}^k}=0.
\end{eqnarray}
We note that notation $\nabla_{\CS^k}  E(\bz ^{k})$  is well defined due to \eqref{supp-wk-S}. However, to ensure a decreasing property of the function values, we also need to find a finite integer $s_k$ such that 
\begin{eqnarray}
    \label{Newton-descent-property-1}
	&& F(\bz ^{k}-\gamma^{s_k}\bd^k)\begin{array}{l}  \leq F(\bz ^{k})- \frac{\sigma}{2} \|\gamma^{s_k}\bd^k\|^2,\end{array}
\end{eqnarray}
where $\gamma\in(0,1)$. Finally, update $\beta_k=\gamma^{s_k}$ and $\bx^{k+1}=\bz ^{k}-\beta_k\bd^k$. 
\end{itemize}

In practice, equations \eqref{Newton-descent-property} may be unsolvable. In this case, we go back to  $\bz ^{k}$ and let $\bx^{k+1}=\bw^k$. In addition,  to accelerate the computation, one can execute Newton steps only when condition 
\begin{eqnarray}\label{switch}
\CS^k={\rm supp}(\bx^{k}) 
\end{eqnarray}
is satisfied. The motivation for using such a switch for Newton steps is that at the beginning the algorithm, $|\CS^k|$ may be large, resulting in a large system of equations \eqref{Newton-descent-property} to be solved. However, along with the iteration number rising, the support set of $\bx^{k}$ can be identified to be the same as $\CS^k$ and $\CS^k$ has a small magnitude, then performing a Newton step enables very fast convergence. Nevertheless, we design two algorithmic frameworks of the proximal Newton pursuit in Algorithms \ref{algorithm 1} and \ref{algorithm 2}. The only difference between them is whether condition \eqref{switch} needs to be satisfied to perform the semismooth Newton step in each iteration. 

\begin{algorithm}[th] 
 \caption{{PSNP}: Proximal Semismooth Newton pursuit \label{algorithm 1}} 
\begin{algorithmic}[1]
\State Initialize $\bx^{0}$, $\sigma >0$, $\gamma\in(0,1)$, and set $k = 0$. 
 \For{$k=0,1,2,\ldots$}
\State \textit{Proximal descent:}
\State Find an $\alpha_k$ such that \eqref{armijio-descent-property}  and \eqref{armijio-descent-property-1}.   

 \State Update $\bx^{k+1}=\bz ^{k}$ and $\CS^k={\rm supp}(\bz ^{k}).$
 
\State \textit{Semismooth Newton pursuit:} 
 \State Choose $\M^k\in\partial^2_{\CS^k} E(\bw^k)$.  
  \If{\eqref{Newton-descent-property} is solvable and   \eqref{Newton-descent-property-1} is met}
  \State  $\beta_k=\gamma^{s_k}$ and $\bx^{k+1}=\bz ^{k}-\beta_k\bd^k$. 
  \EndIf
 \EndFor
\end{algorithmic}
\end{algorithm}

\begin{algorithm}[th] 
 \caption{{PCSNP}: Proximal Conditional Semismooth Newton pursuit \label{algorithm 2}}
\begin{algorithmic}[1]
\State Initialize $\bx^{0}$, $\sigma >0$, $\gamma\in(0,1)$, and set $k = 0$. 
 \For{$k=0,1,2,\ldots$}
\State \textit{Proximal descent:} 
\State Find an $\alpha_k$ such that \eqref{armijio-descent-property}  and \eqref{armijio-descent-property-1}.   

 \State Update $\bx^{k+1}=\bz ^{k}$ and $\CS^k={\rm supp}(\bz ^{k}).$
 \State \textit{Conditional Semismooth Newton pursuit:}
  \If{Condition \eqref{switch} is met}
 \State Choose $\M^k\in\partial^2_{\CS^k} E(\bw^k)$.  
  \If{\eqref{Newton-descent-property} is solvable and   \eqref{Newton-descent-property-1} is met}
  \State  $\beta_k=\gamma^{s_k}$ and $\bx^{k+1}=\bz ^{k}-\beta_k\bd^k$. 
  \EndIf
    \EndIf
 \EndFor
\end{algorithmic}
\end{algorithm}

\begin{remark}\label{remark-alg} Regarding Algorithms \ref{algorithm 1} and \ref{algorithm 2}, we have the following comments.
\begin{itemize}
\item[A.] A practical way to find an $\alpha_k$ that satisfies
 condition (\ref{armijio-descent-property-1}) is to adopt the  Armijo line search. More precisely, for given constants $\tau>0$ and $\gamma\in(0,1)$, let $\alpha_k = \tau\gamma^{t_{k}}$, where $t_k = 0, 1, \ldots$  is the smallest integer such that $\alpha_k$ satisfying condition (\ref{armijio-descent-property-1}).  In the next sub-section, we will show that the generated sequence can converge for any $\alpha$ satisfying (\ref{def-alpha-k}). This condition (\ref{def-alpha-k}) can be ensured if $\alpha_k$ is generated by the Armijo line search. In fact, we can always choose
\begin{eqnarray*}  
\begin{array}{l}t_k \equiv \left\lceil \log_{\gamma}\Big(\frac{1}{\tau(\sigma + \sigma_0)}\Big) \right\rceil,~~\alpha_k = \tau\gamma^{t_k}, ~~\forall~k\geq 0,
\end{array}\end{eqnarray*} 
  where $\lceil a\rceil$  is the ceiling of $a$ and $\sigma_0$ is defined in (\ref{def-consts}), then (\ref{def-alpha-k}) is achieved.  A similar scheme is also used to search $\beta_k$. 
 
  \item[B.] To ensure  the solvability of equations (\ref{Newton-descent-property}), we expect $\M^k  =   \H^k +  \lambda \nabla^2 \|\bw^k_{\CS^k}\|^q_q$ is non-singular, where $\H^k\in\partial^2_{\CS^k}f(\bw^k)$.
For $q=0$,  it is non-singular if $\H^k$ is non-singular.  
For any $q\in(0,1)$, as $\nabla_{\CS^k}^2\|\bw^k\|^q_q$ is diagonal and non-singular, $\M^k$ is non-singular if all eigenvalues of matrix $(\nabla^2\|\bw^k_{\CS^k}\|^q_q)^{-1}\H^k$ are not equal to $-\lambda$. This scenario often occurs with high probability in numerical computations. Additionally, solving (\ref{Newton-descent-property}) is computationally efficient when $|\CS^k|\ll n$. Alternatively, in the case of a larger $|\CS^k|$ (which may occur at the beginning of the algorithms), we can employ the conjugate gradient method for resolution.
\item[C.] As we mentioned before, the only difference between  {PSNP} or {PCSNP} is whether condition (\ref{switch}) needs to be satisfied to perform the semismooth Newton step in each iteration. It is evident that the computational complexity of {PCSNP} is lower than {PSNP} when condition (\ref{switch}) is not satisfied. Nevertheless, according to Theorem \ref{lemma-accumulating-convergence}, condition (\ref{switch}) can be always guaranteed eventually, which indicates that {PCSNP} and {PSNP} tend to be the same.

 Fig. \ref{fig:self-com} demonstrates the performance of PSNP and PCSNP in solving the CS problem. For each $q$, PSNP   takes fewer iterations but longer computational time than PCSNP. This is reasonable because the former benefits from Newton steps more frequently than the latter.
 \end{itemize}
\end{remark}
  \begin{figure}[t]
\centering
\includegraphics[scale=0.95]{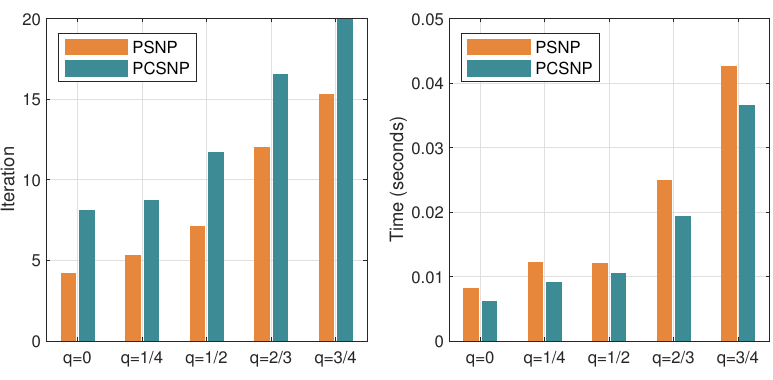} 
\caption{Comparison between PSNP and PCSNP solving CS problems. \label{fig:self-com}}
\end{figure}

\subsection{Convergence analysis}\label{sec:global-convergence}
In this subsection, we aim to establish the main convergence results. For a given $\bx^0$ and a scalar $\tau>0$, we first define some constants:
\begin{eqnarray}
\Omega^0 &:=&\{\bx\in\R^n: F(\bx)\leq F(\bx^0)\},\nonumber\\
\label{def-consts}c_0 &:=& {\sup}_{\bx\in \Omega^0}  5 \left( \| \bx\|^2 +  \tau ^2\|\nabla f(\bx)\|^2 +     \tau \lambda \|\bx\|^q_q \right),\\
\sigma_0&:=& {\sup}_{\bx\in{\mathbb N}( 0,c_0)}{\sup}_{\H\in\partial^2 f(\bx)} \|\H\|.\nonumber
\end{eqnarray}
We first present a sufficient condition to ensure the boundedness of $\sigma_0$, which is associated with a coercive function. We say $f$  is coercive if $\lim_{\|\bx\|\to\infty}f(\bx)=\infty$.
\begin{assumption}\label{assumptions-f}  $f$ is $LC^1$ on ${\mathbb N}( 0,c_0)$. Moreover, $f$ is coercive if $q=0$.
\end{assumption}
\begin{lemma} Under Assumption \ref{assumptions-f}, $\sigma_0<\infty$.
\end{lemma}
\begin{proof} For $q=0$, since $f$ is bounded from below,  $\|\cdot\|_0$ is bounded, and $f$ is coercive, we have $F(\cdot)=f(\cdot)+\lambda\|\cdot\|_0$ is coercive.  For $q\in(0,1)$, since $\|\cdot\|_q^q$ is coercive and $f$ is bounded from below, we have $F(\cdot)=f(\cdot)+\lambda\|\cdot\|_q^q$ is coercive. 
The coerciveness suffices to the boundedness of  $\Omega^0 $, thereby implying the boundedness of $c_0$, which by the local Lipschitz continuity of $\nabla f$ enables the boundedness of $\sigma_0$, see \cite[Sect. 2 ($P_1$)]{hiriart1984generalized}.
\end{proof}
Therefore, under Assumption \ref{assumptions-f} we can choose a universal $\alpha_k$ as
\begin{eqnarray}\label{def-alpha-k}
 \eqspace{1}
\begin{array}{l}
\alpha_k \equiv \alpha \in \left(0, \min\left\{\tau, \frac{1}{\sigma + \sigma_0}\right\}\right],~~\forall~k\geq 0.
\end{array}\end{eqnarray} 
In the sequel, all the convergence results are established based on $\alpha_k$ chosen as \eqref{def-alpha-k}. We will not mention it for simplicity unless further explanations are needed.   Our first result shows that the above choice of $\alpha_k$ enables conditions \eqref{armijio-descent-property}  and \eqref{armijio-descent-property-1} to be satisfied so as to ensure a decreasing property of sequence $\{F(\bx^{k})\}$.
\begin{lemma} \label{theo05}   Let  $\{(\bw^k,\bx^{k})\}$ be the sequence generated by  {PSNP} or {PCSNP}. Under Assumption \ref{assumptions-f}, condition (\ref{armijio-descent-property-1}) can be ensured by (\ref{def-alpha-k}) and
\begin{eqnarray}\label{descent-lemma-property}
 \eqspace{1}
\begin{array}{lll}
F(\bx^{k+1})   \leq  F(\bx^k) -  \frac{\sigma}{4} \max\left\{ \|\bz ^{k}-\bx^k \|^2,  \|\bx^{k+1}-\bx^{k}\|^2\right\}.
\end{array}
\end{eqnarray}  
\end{lemma}
\begin{proof} 

 We prove the conclusion by induction. 
 
\underline{For $k=0$}, by denoting $\Delta^0:=\bz ^{0}-\bx^0$, it follows  from \eqref{armijio-descent-property} and $\alpha_0=\alpha$ that  
\begin{eqnarray*} 
\begin{array}{lll}
\frac{1}{2}\| \bz ^{0} -( \bx ^0 - \alpha   \nabla f(\bx^0))\|^2 + \alpha   \lambda \|\bz ^{0}\|^q_q \leq \frac{1}{2}\|   \alpha    \nabla f(\bx^0) \|^2 + \alpha   \lambda \|\bx^0\|^q_q ,
 \end{array} \end{eqnarray*}
 which results in
   \begin{eqnarray} \label{descent-lemma-fact-1}
 \begin{array}{lll}  \langle   \nabla f(\bx^0),  \Delta^0 \rangle + \lambda \|\bz ^{0}\|^q_q \leq -  \frac{1}{2\alpha   }\| \Delta^0 \|^2 +  \lambda \|\bx^0\|^q_q .
 \end{array}  \end{eqnarray}
 This condition also implies that
   \begin{eqnarray*} 
 \begin{array}{lll} \frac{1}{2\alpha   }\| \Delta^0\|^2 \leq \alpha   \|\nabla f(\bx^0)\|^2 +   \frac{1}{4\alpha   } \|\Delta^0\|^2   +  \lambda \|\bx^0\|^q_q 
 \end{array}  \end{eqnarray*}
by a fact $2\langle\bx,\bz\rangle \leq \pi\|\bx\|^2+(1/\pi)\|\bz\|^2$ for any $\pi>0$, thereby giving rise to
     \begin{eqnarray*} 
 \begin{array}{lll} \frac{1}{4}\| \Delta^0\|^2 \leq \alpha^2   \|\nabla f(\bx^0)\|^2     +  \alpha \lambda \|\bx^0\|^q_q\leq \tau^2   \|\nabla f(\bx^0)\|^2     +  \tau \lambda \|\bx^0\|^q_q.
 \end{array}  \end{eqnarray*}
due to  $\alpha \leq \tau$. Using the above condition, we can bound $\bz ^{0}$ by
   \begin{eqnarray} \label{descent-lemma-fact-11}
\| \bz ^{0}\|^2& \leq&
5\| \bx^{0}\|^2 + ({5}/{4}) \|\Delta^0\|^2 \nonumber\\
 &\leq&    5 \| \bx^{0}\|^2 +  5 \tau^2 \|\nabla f(\bx^0)\|^2 +   5 \tau   \lambda \|\bx^0\|^q_q  \leq  c_0,
\end{eqnarray} 
where the first inequality is from $\|\bz^0\|^2  \leq (1+\pi)\|\bx^0\|^2+(1+1/\pi)\|\Delta^0\|^2$ for any $\pi>0$ and the last one holds because of  \eqref{def-consts} and $\bx^0\in\Omega^0$. Since \eqref{descent-lemma-fact-11} and $\|\bx^0\|^2<c_0$ from \eqref{def-consts}, there is $\|\bfz^0\|^2<c_0$ for any  $\bfz^0\in{\rm co}\{\bx^0,\bw^0\}$ and thus  $\|\H^0\|\leq \sigma_0$ for any $\H^0\in\partial^2 f(\bfz^0)$ from \eqref{def-consts}, which by \eqref{MVT} gives rise to
\begin{eqnarray}\label{descent-lemma-fact-2}
F(\bz ^{0})
  &=&    f(\bx^0)+
  \langle  \nabla f(\bx^0),\Delta^0\rangle
+  (1/2) \langle \H^0 \Delta^0, \Delta^0 \rangle+  \lambda\|\bz ^{0}\|^q_q  \nonumber\\
&\leq&    f(\bx^0)+
  \langle  \nabla f(\bx^0),\Delta^0\rangle
+  ({\sigma_0}/{2})\|\Delta^0\|^2 +  \lambda\|\bz ^{0}\|^q_q  \nonumber\\ 
& {\leq}& f(\bx^0) + \lambda\|\bx^0\|^q_q - \left(  {1}/{(2\alpha)} -   {\sigma_0}/{2}  \right)\|\Delta^0\|^2\nonumber\\
&\leq& F(\bx^0)  -  ({\sigma}/{2}) \|\Delta^0\|^2, 
\end{eqnarray}
where the last two inequalities are due to  \eqref{descent-lemma-fact-1} and \eqref{def-alpha-k}. This verifies condition (\ref{armijio-descent-property-1}). Now by   Algorithm \ref{algorithm 1}, if equations \eqref{Newton-descent-property} are solvable  and  $\bd^0$ satisfies  \eqref{Newton-descent-property-1}, or by  Algorithm \ref{algorithm 2}, if condition \eqref{switch} is met, equations \eqref{Newton-descent-property} are solvable,  and  $\bd^0$ satisfies  \eqref{Newton-descent-property-1}, then $\bx^{1}= \bz ^{0}-\beta_0\bd^0$, leading to
\begin{eqnarray}\label{F-x-w-d}
F(\bx^{1})- F(\bx^0) &=& F(\bz ^{0}-\beta_0\bd^0) - F(\bx^0)\nonumber\\
 &\leq& F(\bz ^{0})  -  ({\sigma}/{2}) \|\bx^{1} -\bz ^{0}\|^2- F(\bx^0)\nonumber \\
&\leq&  -  ({\sigma}/{2})  \|\bz ^{0}-\bx^0 \|^2 -  ({\sigma}/{2})  \|\bx^{0+1}-\bz ^{0}\|^2\nonumber\\
&\leq&  -  ({\sigma}/{2})  \|\bz ^{0}-\bx^0\|^2~~\text{or}~~ -   ({\sigma}/{4}) \|\bx^{1}-\bx^{0}\|^2, 
\end{eqnarray}
where the first inequality is from \eqref{Newton-descent-property-1} and the last one is due to $\|\ba-\bb\|^2\leq 2\|\ba \|^2 + 2\|\bb|^2$. 
Otherwise, $\bx^{1}=\bz ^{0}$, which by \eqref{descent-lemma-fact-2} calls forth $F(\bx^{1})  
   -F(\bx^0)  \leq -  \frac{\sigma}{2} \|\bx^{1}-\bx^{0}\|^2.$
Therefore, both cases lead to
\begin{eqnarray}\label{descent-lemma-fact-3}
 \eqspace{1}
\begin{array}{lll}
F(\bx^{1})   \leq  F(\bx^0) - \frac{\sigma}{4} \max\left\{ \|\bz ^{0}-\bx^0 \|^2,  \|\bx^{1}-\bx^{0}\|^2\right\}.
\end{array}
\end{eqnarray}

\underline{For $k=1$}, conditions  \eqref{descent-lemma-fact-2} and \eqref{descent-lemma-fact-3} mean that both $\bz ^0$ and $\bx^1\in \Omega^0$, leading to $\| \bx ^{1}\|^2  < c_0$ from \eqref{def-consts}. Then same reasoning to show  \eqref{descent-lemma-fact-11} enables us to derive $  \| \bz ^{1}\|^2  \leq c_0$, which by the same reason to show \eqref{descent-lemma-fact-2} allows us to prove that $ F(\bz ^{1}) \leq F(\bx^1)  -  \frac{\sigma}{2} \|\bz ^{1}-\bx^1\|^2$ under \eqref{def-alpha-k}.
So (\ref{armijio-descent-property-1}) is satisfied. Again considering two cases $\bx^2=\bz ^1-\beta_1\bd^1$ or $\bx^2=\bz ^1$ leads to
\begin{eqnarray}\label{descent-lemma-fact-5} 
 \eqspace{1}
\begin{array}{lll}
F(\bx^{2})   \leq  F(\bx^1) - \frac{\sigma}{4} \max\left\{ \|\bz ^{1}-\bx^1 \|^2,  \|\bx^{2}-\bx^{1}\|^2\right\}.
\end{array}
\end{eqnarray}

\underline{For $k\geq2$},  repeating the above steps can show \eqref{descent-lemma-property}.
\end{proof}
\begin{remark}From the proof of the above descent lemma, we restrain sequence $\{\bx^k\}$ in a bounded level set $\Omega^0$ under Assumption \ref{assumptions-f}. Then the semismoothness of $\nabla f$  and (\ref{MVT}) lead to the descent property of  $\{F(\bx^k)\}$.  Hence,  the establishment does not necessitate the strong smoothness on the whole space. However,  the (restricted) strong smoothness of $f$ or Lipschitz continuity of $\nabla f$ on the whole space have been frequently assumed to achieve similar results for sparse optimization \cite{Beck13,bahmani2013greedy, zhou2021newton,wangaextended,zhou2021global}. In this regard,
we achieve the same result in a more relaxed condition.
\end{remark}
 Based on the above lemma, we have the following global convergence regarding the sequence of the objective function values and the accumulation point convergence.
\begin{lemma}\label{sequence-convergence}  Let  $\{(\bw^k,\bx^{k})\}$ be the sequence generated by  {PSNP} or {PCSNP}. If  Assumption \ref{assumptions-f} holds, then $\{F(\bx^{k})\}$ converges, $\bz ^{k}-\bx^k\to 0,~ \bx^{k+1}-\bx^{k}\to0$, and any accumulating point of $\{\bx^{k}\}$ is a  P-stationary point $\bx^*$ with $\alpha$ given as (\ref{def-alpha-k}). 
\end{lemma}
\begin{proof} Both  $f$ and $\|\cdot\|^q_q$ are bounded from below, so is $F$, which by \eqref{descent-lemma-property} yields the convergence of $\{F(\bx^{k})\}$. Taking the limit of both sides of \eqref{descent-lemma-property} leads to $ \bz ^{k}-\bx^k \to0$ and $\bx^{k+1}-\bx^{k}\to 0$. Let $\bx^*$ be any accumulating point of $\{\bx^k\}$.  Algorithm \ref{algorithm 1} implies
\begin{eqnarray}\label{u-l-P}
\bz ^{k}   \in \prox( \bx^k -\alpha  \nabla f(\bx^k) ).\end{eqnarray}
By passing to a  subsequence, taking the limit of both sides of \eqref{u-l-P}  derives 
\begin{eqnarray}\label{x*-P-stationary} \bx^{*} \in {\rm Prox}_{\alpha \lambda\|\cdot\|^q_q}( \bx^* -\alpha  \nabla f(\bx^*) ),\end{eqnarray} 
due to \cite[Theorem 1.25]{RW1998}, namely,   $\bx^*$ is a P-stationary point. 
\end{proof}
Now we are ready to give the sequence convergence as follows. 
\begin{theorem} \label{lemma-accumulating-convergence}  Let  $\{(\bw^k,\bx^{k})\}$ be the sequence generated by  {PSNP} or {PCSNP}. The following results hold under Assumption \ref{assumptions-f}.
\begin{enumerate}
\item[1)] {\bf(Support identification)} There is a finite $k_0$ such that $$\CS^k \equiv \supp(\bx^k)\equiv\CS^*,~~k\geq k_0.$$
\item[2)] {\bf(Sequence convergence)} If second-order sufficient condition  (\ref{2nd-sufficient-cond}) holds at an accumulating point  $\bx^*$, then whole sequence $\{\bx^{k}\}$ converges to this P-stationary point $\bx^*$, which is also a unique local minimizer.
\end{enumerate}
\end{theorem}
\begin{proof} 1)   Conditions \eqref{u-l-P} and \eqref{x*-P-stationary} and Proposition \ref{proximal-phi} lead to
 \begin{eqnarray}\label{lower-bd-w-x*}
 \eqspace{1}
 \begin{array}{lll}
 |w_i^{k}| \geq c >0,~~\forall i\in\CS^k,\qquad
  |x_i^{*}|\geq c >0,~~\forall i\in\CS^*,
   \end{array}
\end{eqnarray} 
 for any $k\geq 0$, where $c:=c(\alpha\lambda,q)$ is defined in \eqref{prox-a-ct}. We prove the conclusion by showing three facts for any $k\geq k_0$.
 \begin{itemize} 
 \item $\CS^{k} \equiv\CS^{k+1}$. In fact, by Lemma \ref{sequence-convergence}  $\bz ^{k+1}-\bz ^{k} = (\bz ^{k+1}-\bx^{k+1}) + (\bx^{k+1} - \bx^k)+(\bx^k - \bz ^{k})\to 0.$
 Suppose $\CS^k\neq \CS^{k+1}$, then there is an $i$ such that $i\in \CS^k$ and $i\notin\CS^{k+1}$ (or $i\notin \CS^k$ and $i\in\CS^{k+1}$). This leads to $c\leq w_{i}^{k}\leq \|\bz ^{k+1}-\bz ^{k}\|\to 0$, a contradiction.
 \item $\CS^{k} \equiv\CS^*$. In fact, since $\bz ^{k}-\bx^{k}\to 0$ and $\bx^*$ is an accumulating point of $\{\bx^k\}$,  there is a subsequence $\{\bz ^{k}: k\in K\}$ converging to $\bx^*$. Moreover,  we have shown that  $\CS^{k} \equiv\CS^{k+1}$ for any sufficiently large $k$, which by the similar reasoning to the above, $\bz ^{k} \to \bx^*$ as $k(\in K )\to \infty$, and \eqref{lower-bd-w-x*} enables us to prove $\CS^{k} \equiv\CS^*$.
 \item $\CS^{k} \equiv\supp(\bx^k)$. In fact, $\supp(\bx^{k+1})\subseteq\CS^k$ from  Algorithm \ref{algorithm 1}. If there is $i\in \CS^k$ but $i\notin\supp(\bx^{k+1})$, then $0<c\leq w_{i}^{k}\leq \|\bz ^{k}-\bx^{k}\| \to 0$,  a contradiction. 
 \end{itemize}
2) Lemma \ref{sequence-convergence} states that any accumulating point $\bx^*$ is a P-stationary point. Then based on Theorem \ref{second-order-necessary-condition}, it is a unique local minimizer of $\min_{\bx} F(\bx)$ under condition  (\ref{2nd-sufficient-cond}). Hence, it is isolated from \cite[Definition 2.1]{auslender1984stability},  
 which by \cite[Lemma 4.10]{more1983computing} and  $\bx^{k +1}-\bx^ {k} \rightarrow0$ can conclude  whole sequence $\{\bx^ {k}\}$ converges to  $\bx^*$.
 \end{proof}
\begin{remark} In our derivation of sequence convergence in Theorem \ref{lemma-accumulating-convergence}, we leverage the second-order sufficient condition. This condition can be regarded as the locally strong convexity but confined to a subset $\{\bx: \supp(\bx)=\CS^*\}$, which differs from those used in existing work for sparse optimization. For example, in the context of the sparsity-constrained optimization \cite{Beck13,zhou2021global},  achieving the sequence convergence usually assumes some other conditions such as the isolation  \cite{zhou2021global}, (restricted) strong convexity \cite{Beck13,bahmani2013greedy, yuan2017gradient}, or K$\L{}$ property  \cite{bolte2014proximal}. For (\ref{rso}) with $q\in(0,1)$,  in \cite{wu2022globally}, the sequence generated by the regularized Newton algorithm converges to a stationary point under K$\L{}$ property, an assumption relied on the sequence itself, and some other conditions, as shown in Table \ref{tab:com-algs}. 
\end{remark}

\subsection{Local convergence rate}\label{local-quad-rate}
In this section, we aim to show how fast {PSNP} and {PCSNP} converge to its obtained P-stationary point $\bx^*$. Using its support set, we define
\begin{eqnarray}\label{sub-space} 
\begin{array}{lll} 
\S :=& \{\bx\in\R^n: \supp(\bx) = \CS^*\}.
 \end{array}
\end{eqnarray} 
To achieve the convergence rate, we need two conditions as follows.
\begin{assumption}\label{assumptions-f-sc} $f$ is $SC^1$ on ${\mathbb N}( 0,c_0)$. Moreover, $f$ is coercive if $q=0$.
\end{assumption}
\begin{assumption}\label{assumptions-soc} Second-order sufficient condition  (\ref{2nd-sufficient-cond}) holds at  $\bx^*$.
\end{assumption}
Assumption \ref{assumptions-f-sc}  is stronger than Assumption \ref{assumptions-f} due to \eqref{Lip-Smooth-continuous}. According to \eqref{descent-lemma-property}, we must have $F(\bx^*)\leq F(\bx^k)\leq F(\bx^0)$ and thus $\bx^*\in\Omega^0$. This indicates that $\bx^*\in {\mathbb N}( 0,c_0)$ from \eqref{def-consts}. Consequently, Assumption \ref{assumptions-f-sc} implies that $f$ is $SC^1$ at $\bx^*$. Moreover, under Assumption \ref{assumptions-soc}, 
\begin{eqnarray}\label{2nd-sufficient-cond-0}
\qquad\begin{array}{l}
\sigma_* :=  \frac{1}{2}  \inf_{\M^*\in\partial^2_{\CS^*} E(\bx^*)}\inf_{\bv\in \mathbb{B}}   \langle \M^*  \bv, \bv \rangle >0,
\end{array}
\end{eqnarray}
where $\mathbb{B}:=\{\bv\in\R^{|\CS^*|}: \|\bv\|=1\}.$ This fact enables the following result.
\begin{lemma} Suppose Assumptions \ref{assumptions-f-sc} and \ref{assumptions-soc} hold. Then there exists a neighbourhood $\N $ of $\bx^*$ such that 
\begin{eqnarray}\label{2nd-sufficient-cond-new}
\begin{array}{l}
  \inf_{\M\in\partial^2_{\CS^*} E( \bx)} \inf_{\bv\in \mathbb{B}} \langle \M  \bv, \bv \rangle \geq   \sigma_*,~~~\forall~\bx\in  \S \cap\N.
\end{array} 
\end{eqnarray} 
\end{lemma}
\begin{proof} Suppose this is untrue. Then there is a sequence $\{\bfz^k\}(\subseteq\S \cap\N)\to\bx^*$  and $\M^k\in\partial^2 _{\CS^*} E(\bfz^k)$ such that $\inf_{\bv\in\mathbb{B}}\langle \M^k  \bv, \bv \rangle <  \sigma_*$. Under Assumption \ref{assumptions-f-sc}, $f$ is LC$^1$ at $\bx^*$, which by Proposition \ref{pro-bd} indicates $\{\M^k\}$ is bounded and has a convergent subsequence. By passing to the subsequence, we may assume $\M^k\to\M^*$.  From  \cite[Sect. 2 ($P_2$)]{hiriart1984generalized}, we obtain $\M^* \in \partial^2_{\CS^*} E( \bx^{*})$, which implies $\inf_{\bv\in\mathbb{B}}\langle \M^*  \bv, \bv \rangle <  \sigma_*$ and thus
\begin{eqnarray*} 
\begin{array}{lll}
 2\sigma_* = \inf_{\M^*\in\partial^2_{\CS^*} E(\bx^*)}\inf_{\bv\in \mathbb{B}}   \langle \M^*  \bv, \bv \rangle  \leq    \inf_{\bv\in \mathbb{B}} \langle \M^*  \bv, \bv \rangle <  \sigma_*.
\end{array} 
\end{eqnarray*} 
This contradiction immediately shows the conclusion.
\end{proof}
\begin{lemma} Let  $\{(\bw^k,\bx^{k})\}$ be the sequence generated by  {PSNP} or {PCSNP}. Then    
\begin{eqnarray} 
 \eqspace{1}\label{convergence-rate-fact-2}
 \begin{array}{lll}
 \|\bx^k- \bx^* \|\geq \rho_*\| \bz ^k - \bx^* \|. 
 \end{array}
\end{eqnarray}
for sufficiently large $k$ under Assumptions \ref{assumptions-f-sc} and \ref{assumptions-soc}, where $\rho_*>0$ relied on $\bx^*$. 
\end{lemma}
\begin{proof}  We emphasize that all the following statements are made for sufficiently large $k$.  First of all, by Theorem \ref{lemma-accumulating-convergence}  2), we have $\CS^k \equiv \CS^*$.   As $\bz ^k$ satisfies  \eqref{armijio-descent-property} and $\bx^*$ is a P-stationary point with $\alpha$, from Lemma \ref{lemma-grad-gamma-0}, we have
 \begin{eqnarray*} 
 0 &=& \bz ^k_{\CS^k}-(\bx^{k}_{\CS^k}-\alpha  \nabla_{\CS^k} f(\bx^{k}) ) +\alpha  \lambda\bg^k_{\CS^k} ,\\
 0&=& \nabla_{\CS^*} E( \bx^*) = \nabla_{\CS^k} E( \bx^*) =\nabla_{\CS^k} f(\bx^*) + \alpha  \lambda \bg^*_{\CS^k},  
\end{eqnarray*}  
where $\bg^k\in\partial \|\bz^k\|^q_q$ and $\bg^*\in\partial \|\bx^*\|^q_q$,  which results in
 \begin{eqnarray} \label{convergence-rate-fact-1}
{\bf lhs}^k&:=&\bx^k_{\CS^k}- \bx^*_{\CS^k}   -  \alpha  (\nabla_{\CS^k} f (\bx^{k}) - \nabla_{\CS^k} f (\bx^{*}))\\ \nonumber
&= & \bz ^k_{\CS^k} +\alpha  \lambda \bg^k_{\CS^k} - \bx^*_{\CS^k} - \alpha  \lambda\bg^*_{\CS^k}=:{\bf rhs}^k.
\end{eqnarray}
Let $\H^*\in \partial^2_{\CS^*} f(\bx^*)$ and  $\H^k\in \partial^2_{\CS^*} f(\bfz^k)$ where $\bfz^k\in\co\{\bx^k,\bx^*\}$. 
From Proposition \ref{pro-bd} there is $c_*>0$ such that $\|\H^k\|\leq c_*$, leading to $\|{\bf I}- \alpha \H^k\|\leq 1+\alpha c_*\leq 1+\tau c_*$ from \eqref{def-alpha-k}. Now \cite[Theorem 2.3]{hiriart1984generalized} enables us to obtain
\begin{eqnarray*} 
 \eqspace{1} 
 \begin{array}{lll}
 \|{\bf lhs}^k \| = \| ( {\bf I}- \alpha \H^k)  \cdot (\bx^k_{\CS^k}- \bx^*_{\CS^k} ) \|\leq      (1+\tau c_*) \|\bx^k- \bx^* \|.
 \end{array}
\end{eqnarray*}
Let $\bw^k_{\pi}:=\pi \bw^k +(1-\pi)\bx^*$ for some $\pi\in(0,1)$. It follows from \eqref{lower-bd-w-x*} that $|(\bw_{\pi})_i^k|\geq c$ for any $i\in \CS^k$, leading to $1- \alpha\lambda q (1-q) |(\bw_{\pi})_i^k|^{q-2} \geq 1- q/2\geq 1/2$, which by  the Mean Value Theorem derives
  \begin{eqnarray*}
 \eqspace{1}
 \begin{array}{lll}
\|{\bf rhs}^k \| 
=  \|  \int_0^1 ({\bf I}+\alpha  \lambda \nabla^2_{\CS^k} \|\bw^k_{\pi}\|_q^q ) \cdot (\bz ^k_{\CS^k} - \bx^*_{\CS^k}) d\pi\|
\geq (1/2) \|\bz ^k - \bx^*\|.
 \end{array}
\end{eqnarray*}
which by the bound for $\|{\bf lhs}^k \|$  and \eqref{convergence-rate-fact-1} yields \eqref{convergence-rate-fact-2} with $\rho_*=1/(2+2\tau c_*)$.
\end{proof}
\begin{lemma}\label{newton-step-admit} Let  $\{(\bw^k,\bx^{k})\}$ be the sequence generated by  {PSNP} or {PCSNP}. If Assumptions \ref{assumptions-f-sc} and \ref{assumptions-soc} hold, then the Newton step is always admitted eventually with $\beta_k\equiv 1$ by choosing $\sigma\in(0,\sigma_*)$.
\end{lemma}
\begin{proof}
 By Theorem \ref{lemma-accumulating-convergence},  there exists a neighbourhood $\N $ of $\bx^*$ such that
 \begin{eqnarray}\label{xk-wx-Gamma}
  	 \eqspace{1} 
     \begin{array}{lll} 
\bz ^k\in \S\cap\N,~~ \bx^k\in \S\cap\N,~~\CS^k \equiv \supp(\bx^k)\equiv\CS^*.
	 \end{array}
\end{eqnarray} 
As (\ref{2nd-sufficient-cond}) holds at $\bx^*$ and $\bz ^k\in \N $,  \eqref{2nd-sufficient-cond-new} is satisfied. So equations \eqref{Newton-descent-property} are solvable.  

 We then show $\bd^k$ and $\beta_k=1$ satisfy \eqref{Newton-descent-property-1}. As $\bx^*$ is a P-stationary point,   $\nabla_{\CS^*} E( \bx^*)=\nabla_{\CS^*} F(\bx^*)=0$ by Lemma \ref{lemma-grad-gamma-0} and thus  $\nabla_{\CS^k} E( \bz ^{k})=\nabla_{\CS^*} E( \bz ^{k})\to \nabla_{\CS^*} E( \bx^*) =0$ due to  $\bz ^k\to \bx^*$, which together with \eqref{Newton-descent-property} implies  $\bd^k\to 0$ and $\bd^k_{\overline{\CS}^k}=0$. This fact and Proposition \ref{pro-bd-expansion} bring out that for any $\M^k\in \partial^2_{\CS^*} E( \bw^{k})$,  
 \begin{eqnarray}\label{wk-wtk}
  	 \eqspace{1} 
     \begin{array}{lll} 
\M^k\bd^k_{\CS^k}-(\nabla E)'_{\CS^k}(\bz ^{k};\bd^k) = o(\|\bd^k\|).
	 \end{array}
\end{eqnarray} 
 Direct calculation enables us to derive the following chain of inequalities,
 \begin{eqnarray*}
	&& 2F(\bz ^{k}-\bd^k)= 2E( \bz ^{k}-\bd^k;\CS^k) \\
	&=& 2E( \bz ^{k};\CS^k) - 2\langle\nabla  E( \bz ^{k}),  \bd^k \rangle +    E'' (\bz ^{k};\bd^k)+o(\|\bd^k\|^{2})\\
	&=& 2E( \bz ^{k};\CS^k) - 2\langle\nabla_{\CS^k} E( \bz ^{k}),  \bd^k_{\CS^k}\rangle +  \langle (\nabla E)'_{\CS^k}(\bz ^{k};\bd^k),  \bd^k_{\CS^k}\rangle+o(\|\bd^k\|^{2})\\
	&=& 2E( \bz ^{k};\CS^k) - 2\langle\M^k\bd^k_{\CS^k},  \bd^k_{\CS^k}\rangle +  \langle (\nabla E)'_{\CS^k}(\bz ^{k};\bd^k),  \bd^k_{\CS^k}\rangle+o(\|\bd^k\|^{2})\\
	  &=& 2E( \bz ^{k};\CS^k) -  \langle\M^k\bd^k_{\CS^k},  \bd^k_{\CS^k}\rangle +  \langle  \M^k\bd^k_{\CS^k}-(\nabla E)'_{\CS^k}(\bz ^{k};\bd^k),  \bd^k_{\CS^k}\rangle+o(\|\bd^k\|^{2})\\
	  	  &\leq & 2E( \bz ^{k};\CS^k) -  {\sigma_*} \langle   \bd^k_{\CS^k},  \bd^k_{\CS^k}\rangle + o(\|\bd^k\|^{2})\\ 
	 & \leq& 2F(\bz ^{k})- {\sigma} \|\bd^k\|^2, 
\end{eqnarray*} 
where the third and fourth equations hold  due to the \eqref{cdd} and \eqref{Newton-descent-property},  the two inequalities are from \eqref{2nd-sufficient-cond-new} and \eqref{wk-wtk}, and   $\sigma\in(0,\sigma_*).$ The above condition leads to \eqref{Newton-descent-property-1} with $\beta_k=1$. The whole proof is finished.
\end{proof}
\begin{theorem}\label{quadratic-theorem} Let  $\{\bx^{k}\}$ be the sequence generated by  {PSNP} or {PCSNP}. If Assumptions \ref{assumptions-f-sc} and \ref{assumptions-soc} hold and choose $\sigma\in(0,\sigma_*).$  Then the sequence converges superlinearly. If $f$ is $SC^2$ on ${\mathbb N}( 0,c_0)$, then the sequence converges quadratically.
\end{theorem}
\begin{proof} We first have that $\nabla_{\CS^k} E( \bx ^*)=\nabla_{\CS^*} E( \bx ^*)=\nabla_{\CS^*} F( \bx ^*)=0$ by $\CS^k\equiv\CS^*$  and Lemma \ref{lemma-grad-gamma-0}. Since $\beta_k\equiv 1$, it follows $\bx^{k+1}=\bz ^{k} - \bd^k$ and
\begin{eqnarray}\label{linear-rate-fact-30} 
  \M^k(\bx^{k+1}_{\CS^k}-\bx^*_{\CS^k})&=&   \M^k\cdot(\bz ^{k}_{\CS^k} - \bd^{k}_{\CS^k}  -\bx^*_{\CS^k}) \nonumber\\ 
&=&  \M^k\cdot(\bz ^{k}_{\CS^k} -\bx^*_{\CS^k})  - \nabla_{\CS^k} E( \bz ^k)\nonumber\\
&=& \M^k\cdot(\bz ^{k}_{\CS^k} -\bx^*_{\CS^k})  - \nabla_{\CS^k} E( \bz ^k)+\nabla_{\CS^k} E( \bx ^*)\\
&=&  \M^k\cdot(\bz ^{k}_{\CS^k} -\bx^*_{\CS^k})  - (\nabla E)'_{\CS^k} ( \bx ^*; \bw^k-\bx^*)\nonumber\\
&+& (\nabla E)'_{\CS^k} ( \bx ^*; \bw^k-\bx^*)  - \nabla_{\CS^k} E( \bz ^k)+\nabla_{\CS^k} E( \bx ^*)\nonumber\\
&=& o(\|\bw^k-\bx^*\|),\nonumber
\end{eqnarray} 
where the second and last equations are from \eqref{Newton-descent-property}  and  \eqref{Md-Ed}.  By \eqref{2nd-sufficient-cond-new}, there is
\begin{eqnarray} \label{linear-rate-fact-33} 
~~\sigma_*\|\bx^{k+1}-\bx^*\|^2 =\sigma_*\| \bx^{k+1}_{\CS^k}-\bx^*_{\CS^k} \|^2
\leq  \langle \bx^{k+1}_{\CS^k}-\bx^*_{\CS^k},~ \M^k(\bx^{k+1}_{\CS^k}-\bx^*_{\CS^k})\rangle. 
\end{eqnarray} 
which by \eqref{linear-rate-fact-30}  derives that $\|\bx^{k+1}-\bx^*\|=o(\|\bw^k-\bx^*\|)$, thereby delivering
\begin{eqnarray}\label{linear-rate-fact-31} 
\frac{\|\bx^{k+1}-\bx^*\|}{\|\bx^k-\bx^*\| }  &\leq&\frac{\| \bx^{k+1}-\bx^* \|}{\rho_*\|\bw^k-\bx^*\| } =\frac{o(\|\bw^k-\bx^*\|)}{\rho_*\|\bw^k-\bx^*\| } \to 0,  
\end{eqnarray} 
where the first inequality holds due to \eqref{convergence-rate-fact-2}, showing the superlinear convergence rate. If $f$ is SC$^2$,  then from \eqref{Md-Ed-2} and \eqref{linear-rate-fact-30}, we can obtain
\begin{eqnarray*} 
  \M^k(\bx^{k+1}_{\CS^k}-\bx^*_{\CS^k})=O(\|\bw^k-\bx^*\|^2).
\end{eqnarray*} 
Then similar reasoning to show \eqref{linear-rate-fact-31} can prove the quadratic convergence rate.
\end{proof}
\subsection{Application to the case of strongly smooth $f$} In this subsection, we shall see that all theorems established in Sections \ref{sec:global-convergence} and \ref{local-quad-rate} are still valid if we replace Assumption \ref{assumptions-f} and \ref{assumptions-f-sc} by the strong smoothness of $f$.
\begin{lemma}\label{lemma-strong-descent} Let  $\{(\bw^k,\bx^{k})\}$ be the sequence generated by  {PSNP} or {PCSNP}. Suppose $f$ is $\sigma_0$-strongly smooth. Then condition (\ref{armijio-descent-property-1}) can be ensured by (\ref{def-alpha-k}) and condition (\ref{descent-lemma-property}) still holds. 
\end{lemma}
\begin{proof}  It follows  from \eqref{armijio-descent-property} and $\alpha_k=\alpha$ that  
\begin{eqnarray*} 
 \eqspace{1}
\begin{array}{lll}
\frac{1}{2}\| \bz ^{k} -( \bx ^k - \alpha   \nabla f(\bx^k))\|^2 + \alpha   \lambda \|\bz ^{k}\|^q_q  
 \leq \frac{1}{2}\|   \alpha    \nabla f(\bx^k) \|^2 + \alpha   \lambda \|\bx^k\|^q_q ,
 \end{array} \end{eqnarray*}
 which results in
   \begin{eqnarray*}  
 \begin{array}{lll}  \langle   \nabla f(\bx^k),  \bz ^{k}-  \bx^k\rangle + \lambda \|\bz ^{k}\|^q_q \leq -  \frac{1}{2\alpha   }\| \bz ^{k}-  \bx^k\|^2 +  \lambda \|\bx^k\|^q_q .
 \end{array}  \end{eqnarray*}
 By the strong smoothness and the above condition, we have
 \begin{eqnarray}\label{descent-lemma-fact-2-00}
F(\bz ^{k})
&\leq&    f(\bx^k)+
  \langle  \nabla f(\bx^k),\bz ^{k}-\bx^k\rangle
+  ({\sigma_0}/{2})\|\bz ^{k}-\bx^k\|^2 +  \lambda\|\bz ^{k}\|^q_q  \nonumber\\ 
& {\leq}& f(\bx^k) + \|\bx^k\|^q_q - \left(  {1}/{(2\alpha)   } -   {\sigma_0}/{2}  \right)\|\bz ^{k}-\bx^k \|^2\nonumber\\
&\leq& F(\bx^k)  -  ({\sigma}/{2}) \|\bz ^{k}-\bx^k \|^2, 
\end{eqnarray} 
This verifies condition (\ref{armijio-descent-property-1}). Now by   Algorithm \ref{algorithm 1}, if equations \eqref{Newton-descent-property} are solvable  and  $\bd^k$ satisfies  \eqref{Newton-descent-property-1}, then $\bx^{k+1}= \bz ^{k}-\beta_k\bd^k$  and thus 
\begin{eqnarray*}
F(\bx^{k+1})- F(\bx^k) &=& F(\bz ^{k}-\beta_k\bd^k) - F(\bx^k)\nonumber\\
 &\leq& F(\bz ^{k})  -  ({\sigma}/{2}) \|\bx^{k+1} -\bz ^{k}\|^2- F(\bx^k)\nonumber \\
&\leq&  -  ({\sigma}/{2})  \|\bz ^{k}-\bx^k \|^2 -  ({\sigma}/{2})  \|\bx^{k+1}-\bz ^{k}\|^2\nonumber\\
&\leq&  -  ({\sigma}/{2})  \|\bz ^{k}-\bx^k \|^2~~\text{or}~~ -   ({\sigma}/{4}) \|\bx^{k+1}-\bx^{k}\|^2. 
\end{eqnarray*}
Otherwise, $\bx^{k+1}=\bz ^{k}$, which by \eqref{descent-lemma-fact-2-00} yields
\begin{eqnarray*}
F(\bx^{k+1})  
&\leq& F(\bx^k)  -  ({\sigma}/{2})  \|\bx^{k+1}-\bx^{k}\|^2\\
&=&F(\bx^k) -  ({\sigma}/{4})  \|\bz ^{k}-\bx^k \|^2 -   ({\sigma}/{4}) \|\bx^{k+1}-\bx^{k}\|^2.
\end{eqnarray*} 
Therefore, both cases lead to \eqref{descent-lemma-property}.
\end{proof}
Based on the above lemma, we have the following direct corollary.
\begin{corollary}\label{cor-strong-smooth} Let  $\{\bx^{k}\}$ be the sequence generated by  {PSNP} or {PCSNP}. Suppose   $f$ is $\sigma_0$-strongly smooth. Then any accumulating point of $\{\bx^{k}\}$ is a  P-stationary point $\bx^*$ with $\alpha$ given as (\ref{def-alpha-k}). If further assume that second-order sufficient condition  (\ref{2nd-sufficient-cond}) holds at  $\bx^*$ and choose $\sigma\in(0,\sigma_*)$, then whole sequence $\{\bx^{k}\}$ converges to $\bx^*$ superlinearly, which is a unique local minimizer.  If further assume $f$ is $SC^2$ on ${\mathbb N}( 0,c_0)$, then the sequence converges quadratically.
\end{corollary}
\begin{proof} Based on Lemma \ref{lemma-strong-descent}, the same reasoning to show Lemma \ref{sequence-convergence} and Theorem \ref{lemma-accumulating-convergence} allows us to prove the accumulation point and sequence convergence of $\{\bx^{k}\}$. By \eqref{xk-wx-Gamma}, $f$ is strongly smooth (so is $E(\cdot;\CS^k)$), and $\sigma\in(0,\sigma_*)$, we can verify that
 \begin{eqnarray*}
	&& 2F(\bz ^{k}-\bd^k)= 2E( \bz ^{k}-\bd^k;\CS^k) \\
	&=& 2E( \bz ^{k};\CS^k) - 2\langle\nabla  E( \bz ^{k}),  \bd^k \rangle +    \langle \nabla^2 E(\bz ^{k})\bd^k  ,  \bd^k\rangle +o(\|\bd^k\|^{2})\\
	&=& 2E( \bz ^{k};\CS^k) - 2\langle\M^k\bd^k_{\CS^k},  \bd^k_{\CS^k}\rangle +  \langle \M^k\bd^k_{\CS^k},  \bd^k_{\CS^k}\rangle+o(\|\bd^k\|^{2})\\
	  	  &\leq & 2E( \bz ^{k};\CS^k) -  {\sigma_*} \langle   \bd^k_{\CS^k},  \bd^k_{\CS^k}\rangle + o(\|\bd^k\|^{2})\\ 
	 & \leq& 2F(\bz ^{k})- {\sigma} \|\bd^k\|^2. 
\end{eqnarray*}
Therefore, Newton step is always admitted with $\beta_k\equiv 1$. The remaining proof is the same as that of proving Theorem \ref{quadratic-theorem}, and thus is omitted here.
\end{proof}
We would like to highlight that to bound $\sigma_0$ in \eqref{def-consts},  we assume the coerciveness of $f$ when $q=0$ to ensure the boundedness of $\Omega^0$. However, if $f$ is strongly smooth, then any Hessian matrix of $f$ is bounded, thereby enabling bounded $\sigma_0$. Consequently, the requirement for the coerciveness of $f$ is no longer necessary.  
\subsection{Some examples}
\begin{example}[CS problems] \label{exp-cs} The objective function takes the form of 
\begin{eqnarray}\label{obj-CS} f_{cs}(\bx)=(1/2)\|\A\bx-\bb\|^2,
\end{eqnarray} where $\A\in\R^{m\times n}$ and $\bb\in\R^{m}$. 
\end{example}  
\begin{corollary}\label{quadratic-lemma-cs}   For problem (\ref{rso}) with $f=f_{cs}$, let $\bx^*$ be the limit of sequence $\{\bx^k\}$ generated by  {PSNP} or {PCSNP}.  If second-order sufficient condition (\ref{2nd-sufficient-cond}) holds at $\bx^*$ and choose $\sigma\in(0,\sigma_*)$, then the whole sequence converges to $\bx^*$, a unique local minimizer. Moreover, 
\begin{enumerate}
\item[1)] for $q=0$, the sequence converges within finitely many steps. 
\item[2)] for $q\in(0,1)$,  the sequence converges quadratically.
\end{enumerate}
\end{corollary}
\begin{proof} For $q=0$, $\M^k=(\A^\top \A)_{\CS^k}$ and $\nabla_{\CS^k} E( \cdot) =(\A^\top (\A(\cdot)-\bb))_{\CS^k}$, leading to $\M^k (\bx^{k+1}_{\CS^k}-\bx^*_{\CS^k}) =0$ in \eqref{linear-rate-fact-30}, which by \eqref{linear-rate-fact-33} shows $\bx^{k+1}=\bx^*$.  For $q\in(0,1)$, the conclusion follows Corollary \ref{cor-strong-smooth}.\end{proof}
From Corollary \ref{suf-2nd-sufficient-cond}, a sufficient condition to ensure  (\ref{2nd-sufficient-cond}) is $ \lambda_{\min} \left((\A^\top \A)_{\CS^*} \right) > \frac{q}{2\alpha},$ which 
generally is weaker than the so-called $|\CS^*|$-singularity \cite{Beck13} (namely, any $|\CS^*|$ columns of $\A$ are linearly independent), RIP \cite{jiao2015primal}, and SRC \cite{huang2018constructive} since it only imposes the condition on a fixed sub-matrix $\A_{:\CS^*}$ rather than any sub-matrices of $\A$, where $\A_{:\CS^*}$ is the sub-matrix of $\A$ removing all columns  in $\overline{\CS}^*$.

\begin{example}[$L_2$ norm regularized  logistic regression \cite{bahmani2013greedy}]\label{exp-lr} The objective function  is
\begin{eqnarray}\label{obj-LR}
f_{lr}(\bx)=\begin{array}{l}\frac{1}{m}\sum_{i=1}^{m}\left(\ln(1+ e^{\langle\ba_i,\bx\rangle})-b_i\langle\ba_i,\bx\rangle\right)+\frac{\mu}{2}\|\bx\|^2,\end{array}
\end{eqnarray}    
where  $\ba_i \in \mathbb{R}^{n}$ and $b_i \in \{0,1\},i=1,2,\ldots,m$ are samples and $\mu>0$. 
\end{example}
This function is strongly smooth and strongly convex  \cite[Lemma 2.2-2.4]{wang2019greedy}, and its Hessian is Lipschitz continuous \cite[Proposition 2.1]{wang2021extended}, resulting in the following conclusion from Corollary \ref{cor-strong-smooth}.
\begin{corollary}\label{quadratic-lemma-lr} For problem (\ref{rso}) with $f=f_{lr}$,  let $\bx^*$ be the limit of sequence $\{\bx^k\}$ generated by  {PSNP} or {PCSNP}.  If second-order sufficient condition (\ref{2nd-sufficient-cond}) holds at $\bx^*$ and choose $\sigma\in(0,\sigma_*)$, then the sequence converges to a unique local minimizer quadratically. 
\end{corollary}

\begin{example}[Squared Hinge loss support vector machine (SVM) \cite{chang2008coordinate}] \label{exp-svm} The objective function is
  \begin{eqnarray}\label{obj-svm} 
f_{svm}(\bx)=\begin{array}{l}\frac{1}{2}\left(\sum_{i=1}^{n-1} x_i^2 + cx_n^2  \right) +\frac{\mu}{2m}\sum_{i=1}^{m}\max\left\{1-y_i \left\langle (\ba_i^\top~ 1)^\top,\bx \right\rangle, 0\right\}^2, \end{array}
\end{eqnarray}    
where  $\ba_i \in \mathbb{R}^{n-1}$ and $y_i \in \{1,-1\},i=1,2,\ldots,m$ are samples,  $\mu>0$, and $c$ is a small positive scalar (e.g., $c=0.01$ used in our numerical experiments). 
\end{example}This function is strongly convex and also strongly semismooth everywhere by \cite[Proposition 1.75.]{izmailov2014newton}.  Hence, a direct result of Theorems \ref{lemma-accumulating-convergence} and \ref{quadratic-theorem} is given as follows.
\begin{corollary}\label{quadratic-lemma-svm} For problem (\ref{rso}) with $f=f_{svm}$, let $\bx^*$ be the limit of sequence $\{\bx^k\}$ generated by  {PSNP} or {PCSNP}.  If second-order sufficient condition (\ref{2nd-sufficient-cond}) holds at $\bx^*$ and choose $\sigma\in(0,\sigma_*)$, then the sequence converges to a unique local minimizer quadratically. 
\end{corollary}

Recall Corollary \ref{suf-2nd-sufficient-cond}, when $q=0$ condition \eqref{2nd-neccssary-cond-cs} reduces to  $\lambda_{\min} \left(\H^* \right) > 0. $ We note that  strongly convex functions (e.g., $f_{lr}$ and $f_{svm}$) must satisfy such a condition. Therefore, the second-order sufficient condition in Corollaries  \ref{quadratic-lemma-lr} and \ref{quadratic-lemma-svm} can be removed when $q=0$.

\begin{example}[Quadratic CS problems \cite{shechtman2011sparsity, Beck13, shechtman2014gespar, zhou2022gradient}] \label{QCS} The objective function takes the form of
  \begin{eqnarray}\label{obj-qcs}
f_{qcs}(\bx)=\begin{array}{l}\frac{1}{4m}\sum_{i=1}^{m}(\langle \A_i \bx,\bx\rangle-b_i)^2, \end{array}
\end{eqnarray}    
where  $\A_i \in \mathbb{R}^{n\times n}$ and $b_i \in \R,i=1,2,\ldots,m$. 
\end{example}
This function is not strongly smooth and not convex, but it is twice continuously differentiable and also  $SC^2$ on any bounded region. To satisfy Assumption \ref{assumptions-f}, a sufficient condition to ensure $f_{qcs}$ to be coercive when $q=0$ is $\lim_{\|\bx\|\to\infty}\langle \A_i \bx,\bx\rangle=\infty$ for some $i$. Particularly, this can be guaranteed if $\A_i$ is positive definite.  Hence, a direct result of Theorems \ref{lemma-accumulating-convergence} and \ref{quadratic-theorem} is given as follows.
\begin{corollary}\label{quadratic-lemma-svm} For problem (\ref{rso}) with $f=f_{qcs}$, suppose that $\A_i$ is positive definite for some $i\in\{1,2,\ldots,m\}$ when $q=0$. Let $\bx^*$ be the limit of sequence $\{\bx^k\}$ generated by  {PSNP} or {PCSNP}.  If second-order sufficient condition (\ref{2nd-sufficient-cond}) holds at $\bx^*$ and choose $\sigma\in(0,\sigma_*)$, then the sequence converges to a unique local minimizer quadratically. 
\end{corollary}
\section{Numerical Experiments}\label{sec:num}
In this section, we conduct some numerical experiments to showcase the performance of  {PSNP} and {PCSNP} (available at \url{https://github.com/ShenglongZhou/PSNP})  using MATLAB (R2023b) on a laptop of 64GB memory and Core i9. We initialize $\bx^0=0, \sigma=10^{-4}$ and adopt the  Armijo line search to choose $\alpha_k$ and $\beta_k$, as described in Remark \ref{remark-alg} A. To be more specific,  we fix constant $\tau= 1$ for CS problems and $\tau=10$ for SVM problems, $\gamma =0.5$ and let $\alpha_k = \tau\gamma^{t_{k}}$ and $\beta_k =  \gamma^{s_{k}}$, where $t_k$ and  $s_k$ are the smallest integers in $\{0, 1, 2, \ldots\}$ satisfying conditions (\ref{armijio-descent-property-1}) and (\ref{Newton-descent-property-1}), respectively. 
 
\subsection{Solving CS problems}\label{ssec:cs}
We first demonstrate the performance of {PSNP} for solving the CS problems,
\begin{equation*} 
\begin{array}{l} 
\min_{\bx\in\R^n}~f_{cs}(\bx) +   \lambda \|\bx\|^q_q, \end{array}
\end{equation*} 
where $f_{cs}$ is defined as \eqref{obj-CS} and  $\A\in\R^{m\times n}$ and $\bb\in\R^{m}$ are generated as follows:  Let $\A\in\mathbb{R}^{m\times n}$  be a random Gaussian matrix with each entry identically and independently distributed from standard normal distribution $\mathcal{N}(0,1)$. We then normalize each column of $\A$ to have a unit length. Next, $s$ indices and values of non-zero entries of `ground truth' signal $\bx^*$ are randomly selected from $\{1,2,\ldots,n\}$ and $[-1.5,-0.5]\cup [0.5,1.5]$. Finally, let $\bb=A \bx^* + {\rm nf}\cdot {\boldsymbol \varepsilon}$, where ${\rm nf}$ is the noise ratio and ${\boldsymbol \varepsilon}$ is the noise with $\varepsilon_i\sim\mathcal{N}(0,1)$ for $i=1,2\ldots,m$.

 \subsubsection{Benchmark methods}
There is a large number of algorithms that have been proposed for CS problems. However, our focus is on evaluating the performance of algorithms directly developed to solve problem \eqref{rso}.  Therefore, we select 6 ones shown in Table  \ref{tab:algs} where {IHT$_0$}  is called the iterative hard thresholding \cite{blumensath2008iterative},   {IHT$_{1/2}$} represents the iterative half thresholding \cite{xu2012}, and {TFPC$_{2/3}$} stands for $q$-thresholding fixed point continuation \cite{peng2018global}  (where we fix $q=2/3$). Their main updates take the form of
\begin{eqnarray}\label{PGD}
\bx^{k+1}= {\rm Prox}_{\alpha \lambda\|\cdot\|_q^q} (\bx^{k}-\alpha \nabla f(\bx^{k}) )
\end{eqnarray}
and they correspond this scheme with $q=0, 1/2$, and $2/3$, respectively.   As stated in \cite{blumensath2008iterative, xu2012}, one can choose $\alpha\in(0,  \|\A\|^{-2})$ to preserve the convergence, which however may degrade their performance. So similar to {PSNP}, we employ the Armijo line search to adaptively update $\alpha$. Moreover, we write {PSNP} as {PSNP}$_0$, {PSNP}$_{1/2}$, and {PSNP}$_{2/3}$ if $q=0, 1/2$, and $2/3$, respectively. A similar rule is also applied for PCSNP and {IRucLq}.
Differing from {TFPC$_{2/3}$} increasing $\lambda$ iteratively, we fix $\lambda=0.025(1+q)\|\A^\top \bb\|_{\infty}$ in model \eqref{rso}. Finally, to conduct fair comparisons, we initialize all algorithms by $\bx^0=0$ and terminate them if either $k>10000$ or at the $k$th iteration two conditions $\|\nabla_{\CS^k}  F(\bx^{k})\|_{\infty} <10^{-6}$ and $\CS^{k+1}=\CS^k$ are satisfied.
 
  \begin{table}[th]
	\renewcommand{\arraystretch}{1.0}\addtolength{\tabcolsep}{11pt}
	\caption{Parameters of benchmark methods.}
	\label{tab:algs} 
		\begin{tabular}{lrclrclr}
			\toprule
			\multicolumn{2}{c}{$q=0$}&&\multicolumn{2}{c}{$q=1/2$}&&\multicolumn{2}{c}{$q=2/3$}\\\cmidrule{1-2}\cmidrule{4-5}\cmidrule{7-8} 
		 Algs. & Ref. &&Algs. & Ref. &&Algs. & Ref.  \\\toprule
		 {PSNP}$_0$& Our&&{PSNP}$_{1/2}$& Our&&{PSNP}$_{2/3}$& Our\\ 
		 {PCSNP}$_0$& Our&&{PCSNP}$_{1/2}$& Our&&{PCSNP}$_{2/3}$& Our\\ 
		 {NL0R} & \cite{zhou2021newton}&&{IRucL}$_{1/2}$&\cite{lai2013improved}&&{IRucL}$_{2/3}$&\cite{lai2013improved}\\
		 {IHT}$_0$ & \cite{blumensath2008iterative} &&{IHT}$_{1/2}$& \cite{xu2012}&&{TFPC}$_{2/3}$  & \cite{peng2018global} \\ \botrule
		\end{tabular} 
\end{table}

To evaluate the algorithmic performance, we let $\bx$ be the solution obtained by one algorithm and report three factors: the computational time (in seconds), 
the relative error (ReErr), and the false detection rate (FDR). The latter two are defined by 
$${\rm ReErr}:=   \frac{\|\bx-\bx^*\|}{\|\bx^*\|},\qquad {\rm FDR}:=  \frac{|\{i: x_i=0, x_i^*\neq0\}|}{s}+\frac{|\{i: x_i\neq 0, x_i^*=0\}|}{n-s}.$$ We note that smaller values for these metrics indicate better performance.

\begin{figure}[t]
\centering
\includegraphics[scale=.7]{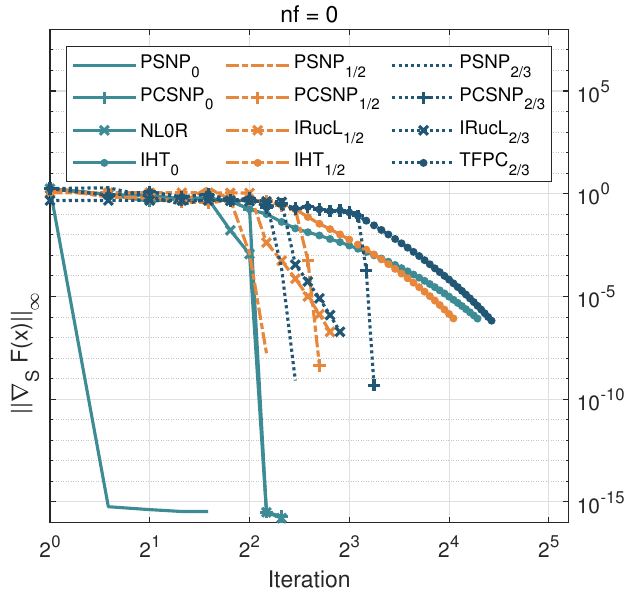}\qquad
\includegraphics[scale=.7]{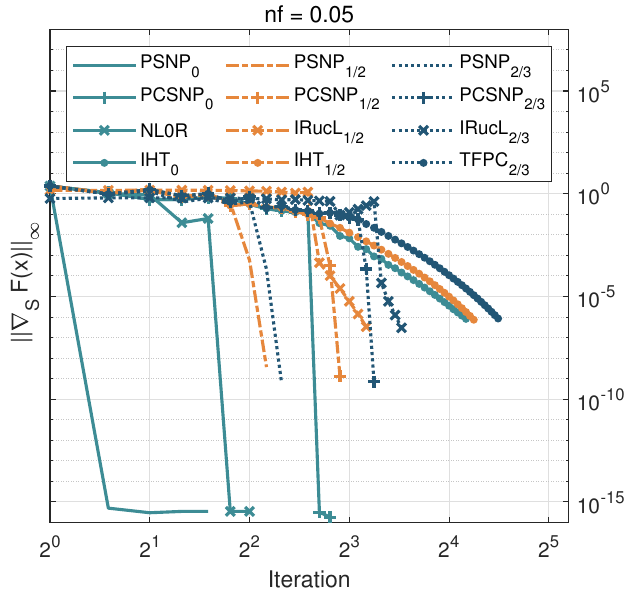}
\caption{Convergence performance. \label{fig:convergence-rate}}
\end{figure}  

 \subsubsection{Numerical comparison}

 Fig. \ref{fig:convergence-rate}  illustrates the convergence performance of each algorithm, where $(m,n,s)=(500,2000,50)$.  For both noise and noiseless settings, $\|\nabla_{\CS^k}  F(\bx^{k})\|_{\infty} $ produced by {PSNP}, PCSNP, {NL0R},  and {IRucLq} drops dramatically when the iteration is over a certain threshold, demonstrating a fast convergence rate. This is because these four algorithms belong to the family of second-order methods (using second-order information, i.e., Hessian). By contrast, {IHT}$_0$, {IHT}$_{1/2}$, and {TFPC}$_{2/3}$ converge slowly.  
Next, we evaluate the efficiency of selected algorithms for solving \eqref{rso} in different scenarios. We alter one factor of $(m,n,s,{\rm nf})$  to see its effect by fixing the other three factors. 

  \begin{figure}[b]
\centering
\includegraphics[scale=.51]{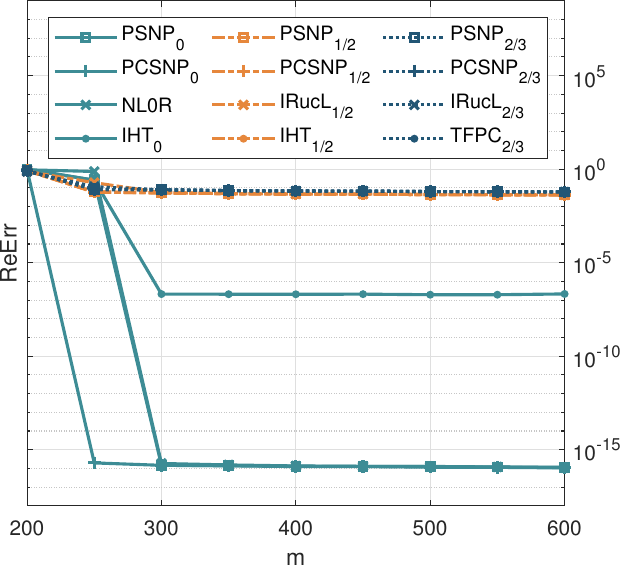}
\includegraphics[scale=.51]{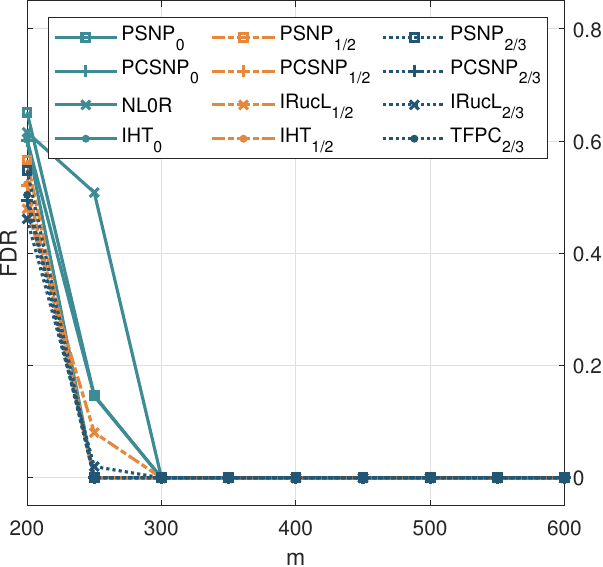}
\includegraphics[scale=.51]{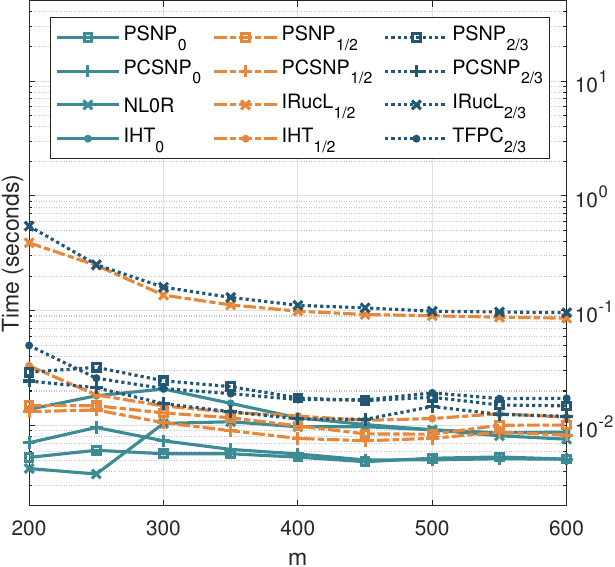}
\caption{Effect of $m$. \label{fig:effect-m}}
\vspace{-3mm}
\end{figure}

 \begin{figure}[b]
\centering
\includegraphics[scale=.51]{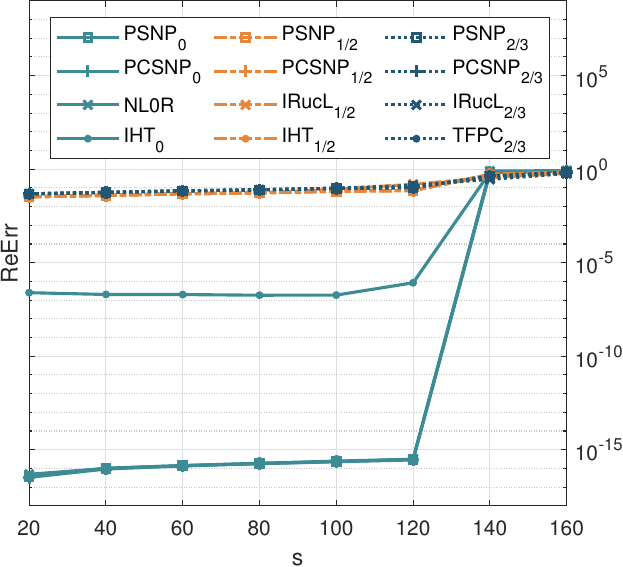}
\includegraphics[scale=.51]{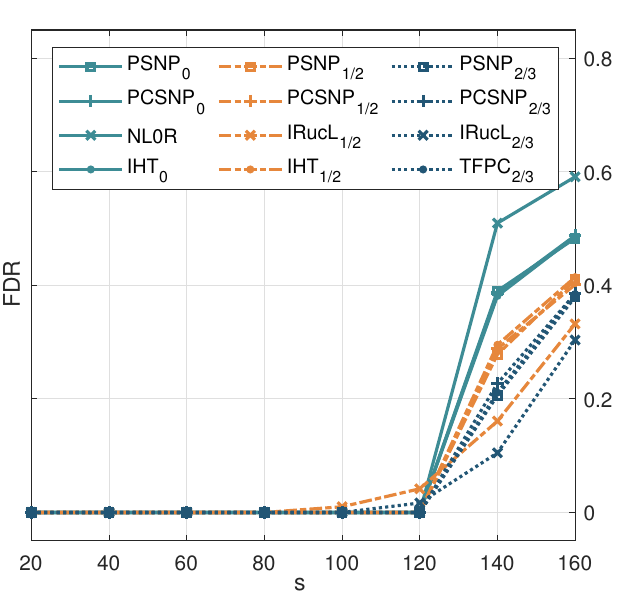}
\includegraphics[scale=.51]{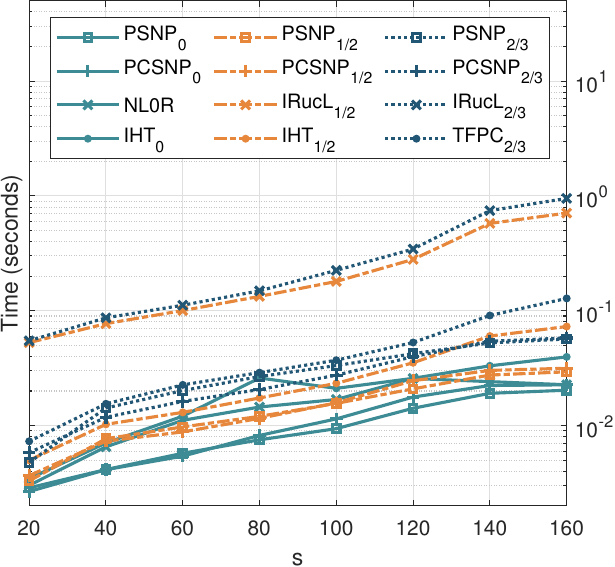}
\caption{Effect of $s$. \label{fig:effect-s}}
\vspace{-3mm}
\end{figure}

 \begin{figure}[t]
\centering
\includegraphics[scale=.51]{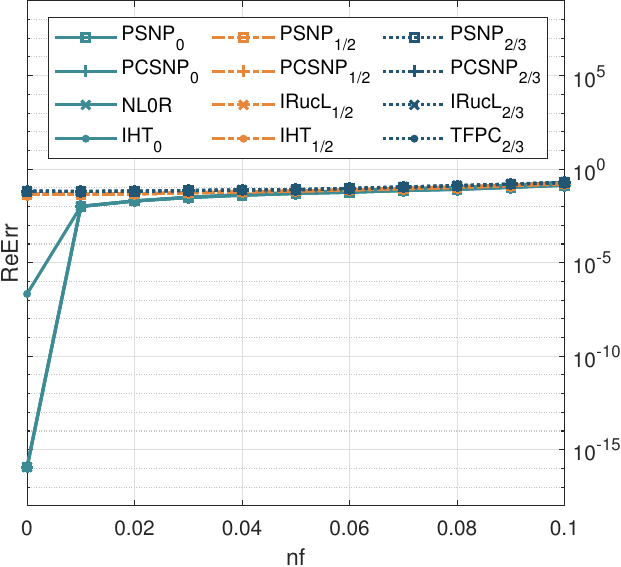} 
\includegraphics[scale=.51]{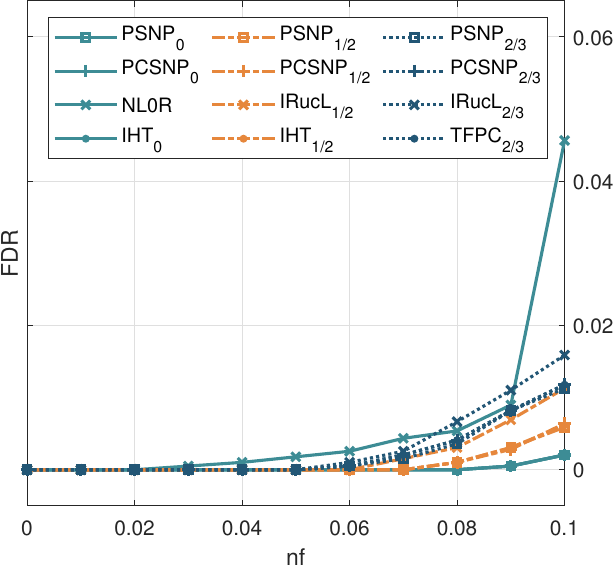}
\includegraphics[scale=.51]{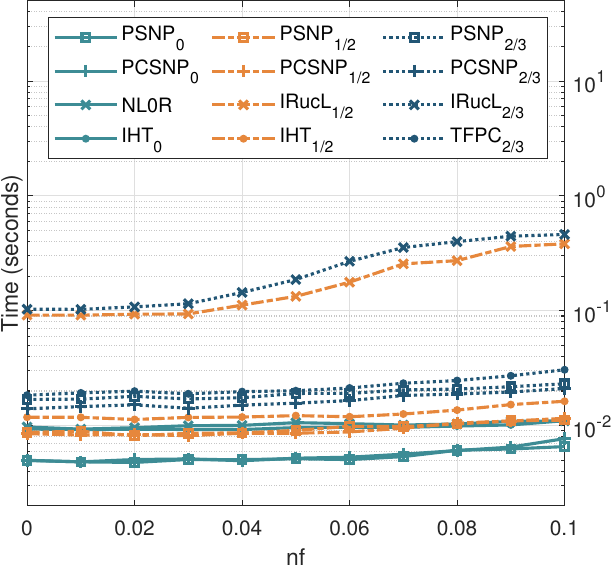}
\caption{Effect of ${\rm nf}$. \label{fig:effect-nf}} \vspace{-3mm}
\end{figure}

\begin{itemize} 
\item[a)] Effect of $m$. To investigate this, we alter ${m\in\{200,300,\ldots,800\}}$ and fix ${(n,s,{\rm nf})=(2000,50,0)}$. Then for each case of $(m,2000,50,0)$, we run 50 trials and report the average results in terms of the median values.  It can be discerned from Fig. \ref{fig:effect-m} that the larger $m$ the easier to detect the `ground truth' signal. Apparently, when $m\geq 300$, {PSNP}$_{0}$, {PCSNP}$_{0}$, and {NL0R} deliver the lowest ReErr. For each $q\in\{0,1/2,2/3\}$, {PSNP} and {PCSNP} run faster than the other algorithms for most scenarios, such as {PSNP}$_{1/2}$ and {PCSNP}$_{1/2}$ perform quicker than IHT$_{1/2}$ and IRucL$_{1/2}$.

\item[b)] Effect of $s$. We fix $(m,n,{\rm nf})=(500,2000,0)$ but alter $s\in\{20,40,\ldots,100\}$.  The median values over 50 trials are presented in  Fig. \ref{fig:effect-s}. It is evident that the larger $s$ the harder to recover the signal. In particular, when $s\leq120$ {PSNP}$_{0}$, {PCSNP}$_{0}$, and {NL0R} generate  the lowest ReErr. However, when $s>120$, the positive FDR indicates that all algorithms are unable to identify the support set of the true signal. For each $q\in\{0,1/2,2/3\}$, {PSNP} and {PCSNP} run faster than the other two algorithms, such as {PSNP}$_0$ and {PCSNP}$_0$ perform quicker than {NL0R} and IHT$_0$.

\item[c)] Effect of {\rm nf}. We fix $(m,n,s)=(500,2000,50)$ but increase ${\rm nf}$ over range $\{0,0.02,\ldots,$ $0.1\}$.  The median results over 50 trials are reported in   Fig. \ref{fig:effect-nf}.  One can observe that {NL0R} is more sensitive to noise. It fails to recover the signals after  {\rm nf} over 0.04 due to DFR$>0$.  By contrast, {PSNP}$_{0}$ and {PCSNP}$_{0}$, and IHT$_0$ are robust to {\rm nf} as they always produce DFR close to $0$. Once again, for each $q\in\{0,1/2,2/3\}$, {PSNP} and {PCSNP} run faster than the other two algorithms.

\item[d)] Effect of a higher dimension. We note that {IRucLq} has to solve an $m\times m$ order linear equation at every step, leading to a long computational time if $m$ is large (e.g., $m\geq20000$). Therefore, 
To test all algorithms solving problems in larger sizes, we generate sparse matrices $\A$ with $1\%$ non-zero entries. We fix $(m,n,s)=(20000,100000,2000)$ and consider noise factor   ${\rm nf}\in\{0, 0.05,0.1\}$. The median results over 20 trials are recorded in Table \ref{tab:higher-dim}. For each case of $(q, {\rm nf})$,  {PSNP}  and {PCSNP} attain the lowest ReErr and FDR and the shortest computational time, demonstrating the superior performance.
\end{itemize}

\begin{table}[th]
	\renewcommand{\arraystretch}{1}\addtolength{\tabcolsep}{1.5pt}
	\caption{Effect of a higher dimension.} 
	\label{tab:higher-dim} 
		\begin{tabular}{lcccccccccccc}
			\toprule
			&\multicolumn{3}{c}{{\rm nf} $=0$}&&\multicolumn{3}{c}{{\rm nf} $=0.05$}&&\multicolumn{3}{c}{{\rm nf} $=0.1$}\\ \cmidrule{2-4}\cmidrule{6-8}\cmidrule{10-12}
	Algs.	&	ReErr	&	FDR	&	Time	&	&	ReErr	&	FDR	&	Time && ReErr	&	FDR	&	Time	\\\midrule
PSNP$_0$	&	0.000 	&	0	&	0.281 	&	&	0.051 	&	0.00e-0	&	0.300 	&	&	0.126 	&	9.05e-3	&	0.468 	\\
PCSNP$_0$	&	0.000 	&	0	&	0.329 	&	&	0.051 	&	0.00e-0	&	0.352 	&	&	0.125 	&	8.77e-3	&	0.553 	\\
NL0R	&	0.000 	&	0	&	0.604 	&	&	0.059 	&	1.86e-3	&	0.545 	&	&	0.172 	&	2.98e-2	&	0.566 	\\
IHT$_0$	&	0.000 	&	0	&	0.821 	&	&	0.051 	&	0.00e-0	&	0.870 	&	&	0.125 	&	8.83e-3	&	1.025 	\\\midrule
PSNP$_{1/2}$	&	0.055 	&	0	&	0.507 	&	&	0.077 	&	0.00e-0	&	0.550 	&	&	0.153 	&	1.04e-2	&	0.670 	\\
PCSNP$_{1/2}$	&	0.055 	&	0	&	0.410 	&	&	0.077 	&	0.00e-0	&	0.479 	&	&	0.156 	&	1.09e-2	&	0.995 	\\
IRucL$_{1/2}$	&	0.056 	&	0	&	30.41 	&	&	0.080 	&	1.51e-3	&	44.89 	&	&	0.175 	&	1.76e-2	&	156.0 	\\
IHT$_{1/2}$	&	0.055 	&	0	&	0.832 	&	&	0.077 	&	0.00e-0	&	0.868 	&	&	0.156 	&	1.09e-2	&	1.318 	\\\midrule
PSNP$_{2/3}$	&	0.080 	&	0	&	0.650 	&	&	0.097 	&	0.00e-0	&	0.758 	&	&	0.181 	&	1.03e-2	&	1.084 	\\
PCSNP$_{2/3}$	&	0.080 	&	0	&	0.536 	&	&	0.097 	&	0.00e-0	&	0.694 	&	&	0.184 	&	1.21e-2	&	1.435 	\\
IRucL$_{2/3}$	&	0.080 	&	0	&	28.95 	&	&	0.098 	&	5.10e-5	&	69.37 	&	&	0.191 	&	1.36e-2	&	187.5 	\\
TFPC$_{2/3}$	&	0.080 	&	0	&	1.009 	&	&	0.097 	&	0.00e-0	&	1.082 	&	&	0.184 	&	1.22e-2	&	2.055 	\\
\botrule
		\end{tabular} 
\end{table}
 \subsection{Solving sparse logistic regression problems}
We then employ the selected algorithms to solve the $L_2$ norm regularized  logistic regression,
\begin{eqnarray}\label{lr-q}
\begin{array}{l}
\min_{\bx\in\R^n}~f_{lr}(\bx) +   \lambda \|\bx\|^q_q,
\end{array}
\end{eqnarray}    
where $f_{lr}$ is defined in Example \ref{exp-lr}. Data $\ba_i$ and labels ${b_i\in\{0,1\}}$ are generated using $8$ selected real datasets whose dimensions are given in Table \ref{Detail-datasets}, where `arcene' and `newgp' are from UCI \cite{asuncion2007uci} and Glmnet  \cite{qian2013glmnet}, respectively, and the remaining is from LIBSVM \cite{chang2011libsvm}. For the five datasets with ${n\leq10000}$, 
sample-wise normalization has been conducted so that each sample has mean zero and variance one, and then the same normalization is employed feature-wisely. For the remaining datasets, they are feature-wisely scaled to $[-1,1]$.  To measure the algorithmic performance, we report the computational time (in seconds), objective function value $f_{lr}$,  size $|\CS|$ of the support set of $\bx$, and the classification accuracy,  \begin{eqnarray*} 
\begin{array}{l}{\rm Acc}:=1- \frac{1}{m}\sum_{i=1}^m|b_i-\sign (\max\{\langle \ba_i,\bx\rangle,0\})|.\end{array}
\end{eqnarray*} 

\begin{table}[t]  
	\renewcommand{\arraystretch}{1}\addtolength{\tabcolsep}{7pt}
\caption{Dimensions of eight real datasets. \label{Detail-datasets} }
\begin{tabular}{llrr|llrr}\toprule  
Dataset&Source&$m$ &$n$ &Dataset&Source&$m$ &$n$ \\ \toprule
{duke}&LIBSVM  &38	&7,129&{gisette}&LIBSVM &1,000&5,000\\ 
{leuke}&LIBSVM &38&7,129&{rcv1}&LIBSVM &	20,242&47,236\\
{colon}&LIBSVM &62 &2,000&{newsgp}&Glmnet  &11,314 &777,811\\
{arcene}&UCI    &100 &10,000&
{news20}&LIBSVM &19,996&1,355,191\\
\hline
\end{tabular} 
\end{table} 

We compare {PSNP} with {IHT$_0$}, {IHT$_{1/2}$}, and {TFPC$_{2/3}$} since the latter three can be readily adapted to solve the SVM problems by just substituting ${f=f_{lr}}$ into \eqref{PGD}. For simplicity, we set  $ \lambda= 0.005/({m\log_2(n)})\|\sum_{i=1}^mb_i\ba_i\|_{\infty}$ for model \eqref{svm-q} and  ${\mu=10\lambda}$ in $f_{lr}$. Moreover, we terminate all algorithms if  either ${k>10000}$ or at the $k$th iteration two conditions ${\CS^{k+1}=\CS^k}$ and ${\|\nabla_{\CS^k}  F(\bx^{k})\|_{\infty} <{\rm ln}({mn})10^{-7}}$ are satisfied.  

  \begin{figure}[!h]
\centering
\includegraphics[scale=0.78]{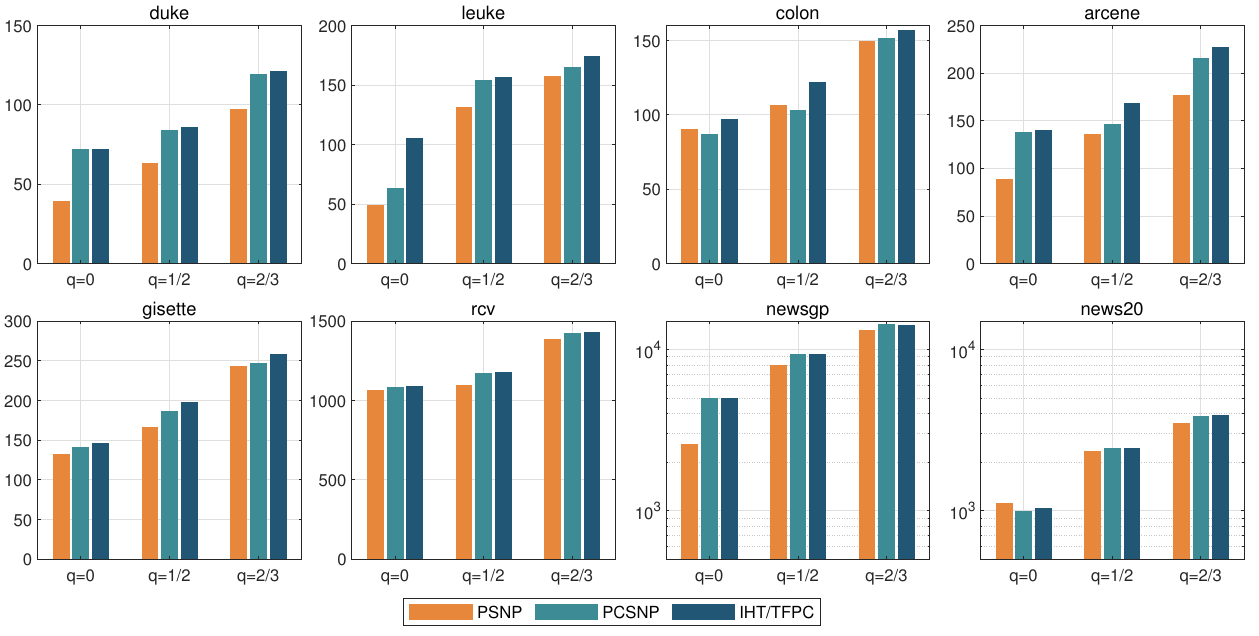} 
\caption{Length of the support sets for logistic regression problems. \label{fig:supp-lr}}
\end{figure}

From Fig. \ref{fig:supp-lr}, one can observe that the smaller $q$ 
enables fewer lengths of the support sets, but higher objective function values and lower accuracy, as reported in Table \ref{tab:LR}. Apparently, for each $q\in\{0,1/2,2/3\}$, PSNP and PCSNP run faster than their counterparts. Taking dataset {news20} as an example, {PSNP}$_{1/2}$, {PCSNP}$_{1/2}$, and  {IHT}$_{1/2}$  respectively consume $3.603$, $9.546$, and $30.49$ seconds, respectively. We point out that the classification accuracy usually relies on the sparsity level of the solution, that is the larger $|\CS|$ is the higher Acc is.

\begin{table}[!t]
	\renewcommand{\arraystretch}{1}\addtolength{\tabcolsep}{-3pt}
	\caption{Performance for solving logistic regression problems.}
	\label{tab:LR} 
		\begin{tabular}{llcccccccccccccc}
			\toprule
	&        & \multicolumn{3}{c}{$q=0$} &&\multicolumn{3}{c}{$q=1/2$} && \multicolumn{3}{c}{$q=2/3$}\\\cmidrule{3-5}\cmidrule{7-9}\cmidrule{11-13}
	&		&	PSNP$_0$	&	PCSNP$_0$	&	IHT$_0$	&	&	PSNP$_{1/2}$	&	PCSNP$_{1/2}$	&	IHT$_{1/2}$	&	&	PSNP$_{2/3}$	&	PCSNP$_{2/3}$	&	TFPC$_{2/3}$	\\\toprule
	&	Acc	&	1.000 	&	1.000 	&	1.000 	&	&	1.000 	&	1.000 	&	1.000 	&	&	1.000 	&	1.000 	&	1.000 	\\
 {duke}	&	$f_{svm}$	&	0.008 	&	0.005 	&	0.005 	&	&	0.007 	&	0.005 	&	0.005 	&	&	0.005 	&	0.004 	&	0.004 	\\
	&	Time	&	0.068 	&	0.092 	&	0.082 	&	&	0.160 	&	0.175 	&	0.183 	&	&	0.212 	&	0.203 	&	0.220 	\\\midrule
	&	Acc	&	1.000 	&	1.000 	&	1.000 	&	&	1.000 	&	1.000 	&	1.000 	&	&	1.000 	&	1.000 	&	1.000 	\\
{leuke}	&	$f_{svm}$	&	0.011 	&	0.007 	&	0.006 	&	&	0.005 	&	0.004 	&	0.004 	&	&	0.004 	&	0.004 	&	0.004 	\\
	&	Time	&	0.062 	&	0.161 	&	0.097 	&	&	0.215 	&	0.351 	&	0.346 	&	&	0.382 	&	0.396 	&	0.435 	\\\midrule
	&	Acc	&	1.000 	&	1.000 	&	1.000 	&	&	1.000 	&	1.000 	&	1.000 	&	&	1.000 	&	1.000 	&	1.000 	\\
 {colon}	&	$f_{svm}$	&	0.019 	&	0.015 	&	0.014 	&	&	0.014 	&	0.014 	&	0.013 	&	&	0.012 	&	0.012 	&	0.012 	\\
	&	Time	&	0.027 	&	0.033 	&	0.038 	&	&	0.250 	&	0.137 	&	0.150 	&	&	0.270 	&	0.259 	&	0.467 	\\\midrule
	&	Acc	&	1.000 	&	1.000 	&	1.000 	&	&	1.000 	&	1.000 	&	1.000 	&	&	1.000 	&	1.000 	&	1.000 	\\
{arcene}	&	$f_{svm}$	&	0.010 	&	0.007 	&	0.007 	&	&	0.007 	&	0.007 	&	0.006 	&	&	0.006 	&	0.005 	&	0.005 	\\
	&	Time	&	0.125 	&	0.393 	&	0.451 	&	&	2.138 	&	1.697 	&	1.855 	&	&	4.144 	&	3.933 	&	4.215 	\\\midrule
	&	Acc	&	1.000 	&	1.000 	&	1.000 	&	&	1.000 	&	1.000 	&	1.000 	&	&	1.000 	&	1.000 	&	1.000 	\\
{gisette}	&	$f_{svm}$	&	0.020 	&	0.018 	&	0.017 	&	&	0.016 	&	0.013 	&	0.013 	&	&	0.012 	&	0.012 	&	0.011 	\\
	&	Time	&	0.408 	&	0.729 	&	1.154 	&	&	2.918 	&	3.693 	&	5.013 	&	&	6.313 	&	5.779 	&	8.031 	\\\midrule
	&	Acc	&	0.959 	&	0.959 	&	0.959 	&	&	0.962 	&	0.963 	&	0.963 	&	&	0.963 	&	0.964 	&	0.964 	\\
 {rcv1}	&	$f_{svm}$	&	0.214 	&	0.214 	&	0.214 	&	&	0.211 	&	0.210 	&	0.210 	&	&	0.208 	&	0.207 	&	0.207 	\\
	&	Time	&	0.143 	&	0.248 	&	1.006 	&	&	3.735 	&	2.507 	&	3.307 	&	&	6.495 	&	3.355 	&	2.406 	\\\midrule
	&	Acc	&	0.960 	&	0.978 	&	0.978 	&	&	0.987 	&	0.991 	&	0.990 	&	&	0.995 	&	0.995 	&	0.995 	\\
{newsgp}	&	$f_{svm}$	&	0.276 	&	0.227 	&	0.227 	&	&	0.201 	&	0.191 	&	0.191 	&	&	0.176 	&	0.173 	&	0.173 	\\
	&	Time	&	0.488 	&	0.862 	&	1.676 	&	&	2.821 	&	4.622 	&	5.893 	&	&	7.938 	&	4.504 	&	8.234 	\\\midrule
	&	Acc	&	0.941 	&	0.949 	&	0.949 	&	&	0.962 	&	0.963 	&	0.963 	&	&	0.967 	&	0.968 	&	0.968 	\\
{news20}	&	$f_{svm}$	&	0.266 	&	0.263 	&	0.261 	&	&	0.231 	&	0.229 	&	0.229 	&	&	0.219 	&	0.217 	&	0.216 	\\
	&	Time	&	0.590 	&	2.873 	&	13.12 	&	&	3.603 	&	9.546 	&	30.49 	&	&	12.29 	&	17.46 	&	22.07 	\\
\hline
		\end{tabular} 
\end{table}

\subsection{Solving SVM problems}
We then employ the selected algorithms to solve the squared Hinge loss SVM problem, 
\begin{eqnarray}\label{svm-q}
\begin{array}{l}
\min_{\bx\in\R^n}~f_{svm}(\bx) +   \lambda \|\bx\|^q_q,
\end{array}
\end{eqnarray}    
where $f_{svm}$ is defined in Example \ref{exp-svm}. Data $\ba_i$ and labels ${y_i\in\{1,-1\}}$ are also generated using these datasets in Table \ref{Detail-datasets}. A similar normalization process is also employed for these datasets. For this problem, we set $ \lambda= 0.4/({m\log_{10}(10+n)})\|\sum_{i=1}^mb_i\ba_i\|_{\infty}$ for model \eqref{svm-q} and  ${\mu=30+10{\rm sgn}(m-1000)}$ in objective function $f_{svm}$. The stopping criteria are set the same as that for logistic regression problems. To evaluate the algorithmic permanence, we also report four metrics: the computational time (in seconds), $f_{svm}$,  $|\CS|$, and the classification accuracy,
\begin{eqnarray*} 
\begin{array}{l}{\rm Acc}:=1- \frac{1}{m}\sum_{i=1}^m|y_i-\sign (\langle \ba_i,\bx\rangle)|.\end{array}
\end{eqnarray*} 
From Fig. \ref{fig:supp}, one can observe that the smaller $q$ 
enables fewer lengths of the support sets, but higher objective function values and lower accuracy, as reported in Table \ref{tab:SVM}. Apparently, for each ${q\in\{0,1/2,2/3\}}$, PSNP and PCSNP run much faster than their counterparts. Taking dataset {news20} as an example, {PSNP}$_0$, {PCSNP}$_{0}$, and  {IHT}$_{0}$  respectively consume $1.22$, $2.354$, and $26.30$ seconds.  

  \begin{figure}[!t]
\centering
\includegraphics[scale=0.78]{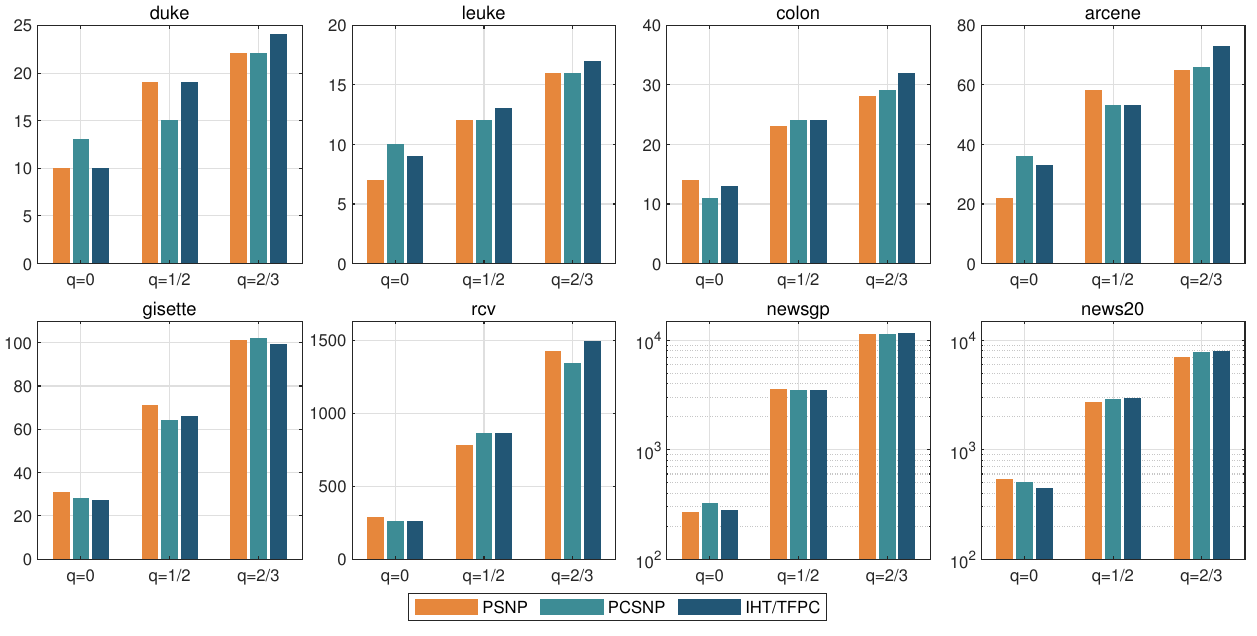} 
\caption{Length of the support sets for SVM problems. \label{fig:supp}}
\end{figure}

\begin{table}[!h]
	\renewcommand{\arraystretch}{1}\addtolength{\tabcolsep}{-3.5pt}
	\caption{Performance for solving SVM problems.}
	\label{tab:SVM} 
		\begin{tabular}{llcccccccccccccc}
			\toprule
	&        & \multicolumn{3}{c}{$q=0$} &&\multicolumn{3}{c}{$q=1/2$} && \multicolumn{3}{c}{$q=2/3$}\\\cmidrule{3-5}\cmidrule{7-9}\cmidrule{11-13}
	&		&	PSNP$_0$	&	PCSNP$_0$	&	IHT$_0$	&	&	PSNP$_{1/2}$	&	PCSNP$_{1/2}$	&	IHT$_{1/2}$	&	&	PSNP$_{2/3}$	&	PCSNP$_{2/3}$	&	TFPC$_{2/3}$	\\\toprule
	&	Acc	&	1.000 	&	1.000 	&	1.000 	&	&	1.000 	&	1.000 	&	1.000 	&	&	1.000 	&	1.000 	&	1.000 	\\
 {duke}	&	$f_{svm}$	&	0.420 	&	0.215 	&	0.318 	&	&	0.148 	&	0.175 	&	0.127 	&	&	0.110 	&	0.108 	&	0.099 	\\
	&	Time	&	0.102 	&	0.115 	&	0.566 	&	&	0.431 	&	0.413 	&	1.022 	&	&	1.264 	&	0.806 	&	2.438 	\\\midrule
	&	Acc	&	1.000 	&	1.000 	&	1.000 	&	&	1.000 	&	1.000 	&	1.000 	&	&	1.000 	&	1.000 	&	1.000 	\\
{leuke}	&	$f_{svm}$	&	0.235 	&	0.140 	&	0.163 	&	&	0.131 	&	0.111 	&	0.102 	&	&	0.095 	&	0.087 	&	0.075 	\\
	&	Time	&	0.047 	&	0.072 	&	1.082 	&	&	0.206 	&	0.256 	&	1.063 	&	&	0.377 	&	0.515 	&	1.916 	\\\midrule
	&	Acc	&	1.000 	&	1.000 	&	1.000 	&	&	1.000 	&	1.000 	&	1.000 	&	&	1.000 	&	1.000 	&	1.000 	\\
 {colon}	&	$f_{svm}$	&	0.720 	&	0.663 	&	0.616 	&	&	0.303 	&	0.313 	&	0.300 	&	&	0.241 	&	0.234 	&	0.213 	\\
	&	Time	&	0.033 	&	0.047 	&	0.179 	&	&	0.180 	&	0.225 	&	0.425 	&	&	0.888 	&	0.466 	&	1.116 	\\\midrule
	&	Acc	&	1.000 	&	1.000 	&	1.000 	&	&	1.000 	&	1.000 	&	1.000 	&	&	1.000 	&	1.000 	&	1.000 	\\
{arcene}	&	$f_{svm}$	&	1.101 	&	0.652 	&	0.611 	&	&	0.249 	&	0.231 	&	0.200 	&	&	0.161 	&	0.164 	&	0.145 	\\
	&	Time	&	0.301 	&	0.651 	&	3.285 	&	&	3.853 	&	5.776 	&	16.63 	&	&	13.99 	&	14.46 	&	25.60 	\\\midrule
	&	Acc	&	0.973 	&	0.966 	&	0.977 	&	&	1.000 	&	1.000 	&	0.999 	&	&	1.000 	&	1.000 	&	1.000 	\\
{gisette}	&	$f_{svm}$	&	1.864 	&	2.011 	&	1.827 	&	&	0.710 	&	0.729 	&	0.708 	&	&	0.483 	&	0.447 	&	0.455 	\\
	&	Time	&	1.025 	&	1.876 	&	11.06 	&	&	3.866 	&	5.860 	&	26.65 	&	&	21.22 	&	10.14 	&	37.90 	\\\midrule
	&	Acc	&	0.924 	&	0.921 	&	0.922 	&	&	0.937 	&	0.936 	&	0.935 	&	&	0.941 	&	0.941 	&	0.941 	\\
 {rcv1}	&	$f_{svm}$	&	11.31 	&	11.44 	&	11.43 	&	&	10.50 	&	10.46 	&	10.46 	&	&	10.25 	&	10.26 	&	10.24 	\\
	&	Time	&	0.340 	&	1.345 	&	3.426 	&	&	1.377 	&	3.888 	&	18.19 	&	&	3.349 	&	4.075 	&	20.79 	\\\midrule
	&	Acc	&	0.699 	&	0.735 	&	0.728 	&	&	0.845 	&	0.844 	&	0.846 	&	&	0.892 	&	0.893 	&	0.893 	\\
{newsgp}	&	$f_{svm}$	&	17.54 	&	17.09 	&	17.20 	&	&	14.37 	&	14.38 	&	14.38 	&	&	12.81 	&	12.81 	&	12.79 	\\
	&	Time	&	0.657 	&	1.685 	&	19.52 	&	&	5.066 	&	7.668 	&	37.54 	&	&	9.920 	&	8.432 	&	68.55 	\\\midrule
	&	Acc	&	0.869 	&	0.868 	&	0.863 	&	&	0.905 	&	0.905 	&	0.906 	&	&	0.921 	&	0.921 	&	0.922 	\\
{news20}	&	$f_{svm}$	&	13.02 	&	13.08 	&	13.23 	&	&	11.41 	&	11.36 	&	11.35 	&	&	10.70 	&	10.64 	&	10.63 	\\
	&	Time	&	1.220 	&	2.354 	&	26.30 	&	&	3.570 	&	10.91 	&	90.15 	&	&	20.60 	&	18.69 	&	116.9 	\\

	\hline
		\end{tabular} 
\end{table}

\section{Conclusion}
In this paper, we provided a unified way to address the $L_q$ norm regularized optimization for any $q\in[0,1)$, where the primary objective function has semismooth gradients. The two proposed algorithms fully leveraged the proximal operator of the $L_q$ norm, which enabled a uniform lower bound. This effectively confined the sequence to a subspace where the semismooth Newton method solved the problem efficiently, thereby leading to the whole sequence convergence at a superlinear or quadratic rate. In contrast to the prevalent use of the K$\L{}$ property to establish the convergence results for the non-smooth and non-convex optimization, we achieved these results based on the second-order sufficient condition.

\backmatter

\bmhead{Acknowledgments} 
This work is supported in part by the Fundamental Research Funds for the Central Universities, the Talent Fund of Beijing Jiaotong University, and the National Natural Science Foundation of China (12001019, 12261020).
%


%

\bibliography{references}

\end{document}